\documentclass[ submission
              , copyright
              , creativecommons
              ]{eptcs}

\def\emails{$^\dag$\texttt{guidoboccali@gmail.com}, $^\oast$\texttt{anlare@ttu.ee}, $^\clubsuit$\texttt{folore@ttu.ee}, and $^\heartsuit$\texttt{stefano.luneia@gmail.com}}
\title{Bicategories of Automata, Automata in Bicategories}
\date{}
\author{Guido \textsc{Boccali}$^\dag$
	\institute{Università di Torino, Torino, Italy}
	\and
	Andrea \textsc{Laretto}$^\oast$
	\institute{Tallinn University of Technology, Tallinn, Estonia}
	\and
	Fosco \textsc{Loregian}$^\clubsuit$
	\institute{Tallinn University of Technology, Tallinn, Estonia\thanks{Loregian was supported by the ESF funded Estonian IT Academy research measure (project 2014-2020.4.05.19-0001).}}% I. Di Liberti provided insight when we felt lost. Innumerable discussions with S. Kasangian were crucial in the development of this work. ChatGPT provided the italicised part of text in \autoref{complain}.}}
	\and
	Stefano \textsc{Luneia}$^\heartsuit$
	\institute{Università di Bologna, Bologna, Italy}
}
\def\titlerunning{Bicategories of Automata, Automata in Bicategories}
\def\authorrunning{Boccali, Laretto, Loregian, Luneia}
\usepackage{lipics-fouche}

\usepackage{quiver}
\usepackage[normalem]{ulem}

\newcommand{\Pageref}[1]{p.~\pageref{#1}}

\NewDocumentCommand{\vxy}{o m}{
	\IfNoValueTF{#1}
	{\vcenter{\xymatrix{#2}}}%   do if \vxy{bla}
	{\vcenter{\xymatrix#1{#2}}}% do if \vxy[foo]{bla}
}
\newcommand{\deferredRef}[1]{The proof is deferred to the appendix, \Pageref{proof_of_#1}}
\NewDocumentCommand{\mealy}{m m m O{I} O{O}}{\vxy{#1 & \ar[l]_-{#2} #1\otimes #4 \ar[r]^-{#3} & #5}}
\NewDocumentCommand{\moore}{m m m O{I} O{O}}{\vxy{#1 & \ar[l]_-{#2} #1\otimes #4 \:\: ; \:\: #1 \ar[r]^-{#3} & #5}}

\def\agdacirc{\firstblank\hspace{-0.25em}\circ\hspace{-0.25em}\firstblank}
\def\Kl{\cate{Kl}}

\def\opFib{\cate{opFib}}
\def\Mly{\cate{Mly}}
\def\bMly{\twoMly^\flat}
\def\Mac{\cate{Mac}^\mathrm{s}}

\def\twoMly{\Mly}%\mymathbb{Mly}}

\def\Rel{\mymathbb{Rel}}
\def\Mre{\cate{Mre}}

\DeclareMathOperator{\id}{id}

\def\pro{\rightsquigarrow}

\usetikzlibrary{decorations,automata,positioning,arrows}
\pgfdeclaredecoration{simple line}{initial}{
	\state{initial}[width=\pgfdecoratedpathlength-1sp]{\pgfmoveto{\pgfpointorigin}}
	\state{final}{\pgflineto{\pgfpointorigin}}
}
\tikzset{
	shift left/.style={decorate,decoration={simple line,raise=#1}},
	shift right/.style={decorate,decoration={simple line,raise=-1*#1}},
}

\def\c{\mathbin{\diamondsuit}}

\def\ls{/}
\def\kons{\mathbin{::}}
\def\concat{\mathbin{\scalebox{.8}{$+\kern-.2em+$}}}
\def\preMac{\cate{Mac}}
\usepackage{secs/moirai}
\hypersetup{
	breaklinks=true,
	pdfencoding=unicode,
	bookmarksnumbered,
	pdfborder={0 0 0},
	pdfauthor={Boccali, Laretto, Loregian, Luneia},
	pdftitle={Bicategories of Automata, Automata in Bicategories},
}
\usepackage{newtxtext}
\usepackage{newtxmath}
\begin{document}
\maketitle

% \todo[inline]{}
\begin{abstract}
    We study \emph{bicategories of} (deterministic) \emph{automata}, %i.e. how automata $E\leftarrow E\otimes I \to O$ organise as the 1-cells of a bicategory $\twoMly_\clK$,
    drawing from prior work of Katis\hyp{}Sabadini\hyp{}Walters, and Di Lavore\hyp{}Gianola\hyp{}Rom\'an\hyp{}Sabadini\hyp{}Soboci\'nski, and linking their bicategories of `processes' to a bicategory of Mealy machines constructed in 1974 by R. Guitart. We make clear the sense in which Guitart's bicategory retains information about automata, proving that Mealy machines \emph{à la} Guitart identify to \emph{certain} Mealy machines \emph{à la} K-S-W that we call \emph{fugal automata}; there is a biadjunction between fugal automata and the bicategory of K-S-W. Then, we take seriously the motto that a monoidal category is just a one-object bicategory. We define categories of Mealy and Moore machines \emph{inside} a bicategory $\bbB$; we specialise this to various choices of $\bbB$, like categories, relations, and profunctors. Interestingly enough, this approach gives a way to interpret the universal property of reachability as a Kan extension and leads to a new notion of 1- and 2-cell between Mealy and Moore automata, that we call \emph{intertwiners}, related to the universal property of K-S-W bicategory.
\end{abstract}
% %=== arxiv abstract
% \begin{comment}
%     We study bicategories of (deterministic) automata, i.e. how automata $E\leftarrow E\otimes I \to O$ organise as the 1-cells of a bicategory $\mathbf{Mly}_\mathcal{K}$, drawing from prior work of Katis-Sabadini-Walters, and Di Lavore-Gianola-Rom\'an-Sabadini-Soboci\'nski, and linking their bicategories of `processes' to a bicategory of Mealy machines constructed in 1974 by R. Guitart. We make clear the sense in which Guitart's bicategory retains information about automata, proving that Mealy machines à la Guitart identify to certain Mealy machines à la K-S-W that we call fugal automata; there is a biadjunction between fugal automata and the bicategory of K-S-W. Then, we take seriously the motto that a monoidal category is just a one-object bicategory. We define categories of Mealy and Moore machines inside a bicategory $\mathbb{B}$; we specialise this to various choices of $\mathbb{B}$, like categories, relations, and profunctors. Interestingly enough, this approach gives a way to interpret the universal property of `reachability of a state' as a Kan extension and leads to a new notion of 1- and 2-cell between Mealy and Moore automata, that we call intertwiners, related to the universal property of K-S-W bicategory.
% \end{comment}
% %===

\section{Introduction}\label{sec:intro}
The profound connection between category theory and automata theory is easily explained: one of the founders of the first wrote extensively about the second \cite{Eilenberg:1974,Eilenberg1967}. A more intrinsic reason is that category theory is a theory of \emph{systems} and \emph{processes}. Morphisms in a category can be considered a powerful abstraction of `sequential operations' performed on a domain/input to obtain a codomain/output. Hence the introduction of categorical models for computational machines has been rich in results, starting from the elegant attempts by Arbib and Manes \cite{17344305293823001d18e665ad1a9d7e3addbe67,arbib1975fuzzy,arbib1975adjoint,Arbib1975,Arbib1980,Pohl_1970} --cf. also \cite{adam-trnk:automata,2009,Ehrig} for exhaustive monographs--
and Goguen \cite{goguen:realization,Goguen1975closed,Goguen1975}, up to the ultra\hyp{}formal --and sadly, under-appreciated-- experimentations of \cite{bainbridge1972unified,Bainbridge1975addressed,guitart1974remarques,guitart1978bimodules,ea309886425a9829ea5df96a513b157fd9940689} using hyperdoctrines, 2-dimensional monads, bicategories, lax co/limits\dots{} up to the modern coalgebraic perspective of \cite{Jacobs2006,Rutten2000,Singh2021,Venema2006}; all this, without mentioning categorical approaches to Petri nets \cite{Meseguer1990}, based essentially on the same analogy, where the computation of a machine is \emph{concurrent} --as opposed to single\hyp{}threaded.

Furthermore, many constructions of computational significance often, if not always, have a mathematical counterpart in terms of categorical notions: the transition from a deterministic machine to a non\hyp{}deterministic one is reflected in the passage from automata in a monoidal category (cf. \cite{Ehrig,Meseguer1975}), to automata in the Kleisli category of an opmonoidal monad (cf. \cite{Guitart1980,jacobs_2016}; this approach is particularly useful to capture categorically \emph{stochastic} automata, \cite{DArgenio2005,arbib1975fuzzy,burroni2009distributive} as they appear as automata in the Kleisli category of a probability distribution monad); \emph{minimisation} can be understood in terms of factorisation systems (cf. \cite{colcombet_et_al,Goguen1975}); behaviour as an adjunction (cf. \cite{2d281e1e3e0b525128f55519cf8a03fa52ce6252,NAUDE1979277}).

The present work starts from the intuition, first presented in \cite{ITA_2002__36_2_181_0,rosebrugh_sabadini_walters_1998}, that the analogy between morphisms and sequential machines holds up to the point that the series and parallel composition of automata should itself be reflected in the `series' and `parallel' composition of morphisms in a category. As a byproduct of the `Circ' construction in \emph{op. cit.}, one can see how the 1-cells of a certain monoidal bicategory specialise exactly a \emph{Mealy machines} $E \xot d E \otimes I \xto s O$ with inputs and outputs $I$ and $O$.

% \medskip
\paragraph{Outline of the paper.} The first result we present in \autoref{sec:bicats} is that this category relates to \emph{another} bicategory constructed by R. Guitart in \cite{guitart1974remarques}. Guitart observes that one can use certain categories $\preMac(\clM,\clN)$ of spans as hom categories of a bicategory $\preMac$, and shows that $\preMac$ admits a concise description as the Kleisli bicategory of the \emph{monad of diagrams} \cite[§1]{guitart1974remarques} (cf. also \cite{ea309886425a9829ea5df96a513b157fd9940689}, by the same author, and \cite{Perrone2022} for a more modern survey); Mealy machines shall be recognisable as the 1-cells of $\preMac$ between monoids, regarded as categories with a single object. The fundamental assumption in \cite{guitart1974remarques} is that a Mealy machine $E \xot d E \otimes M \xto s N$ satisfies a certain property of compatibility with the action of $d$ on $E$, cf. \eqref{beheq}, that we call being a \emph{fugal} automaton:
\[s(e,m\cdot m') = s(e,m)\cdot s(d(e,m),m').\notag\]
This notion can be motivated in the following way: if $s$ satisfies the above equation, then it lifts to a functor $\clE[d]\to N$ defined on the category of elements of the action $d$, and in fact, defines a `relational action' in its own right, compatible with $d$ (formally speaking, $\clE[d]$ is a \emph{displayed category} \cite{lmcs:5252} over $N$).
We show that there is a sub\hyp{}bicategory $\bMly_\Set$ of $\Mly_\Set$ made of fugal automata and that $\bMly_\Set$ is biequivalent (actually, strictly) to the 1-full and 2-full sub\hyp{}bicategory of $\preMac$ spanned by monoids.

The second result we propose in this paper is motivated by the motto for which a monoidal category is just a bicategory with a single object: what are automata \emph{inside a bicategory $\bbB$ with more than one object}, where instead of input/output \emph{objects} $I,O$ we have input/output 1-cells, arranged as $e \overset{\delta}\Leftarrow e\circ i \overset{\sigma}\To o$? Far from being merely formal speculation (a similar idea was studied in a short, cryptic note \cite{Bainbridge1975addressed} to describe behaviour through Kan extensions: we take it seriously and present it as a quite straightforward observation in \autoref{as_a_kan}), we show how this allows for a concise generalisation of `monoidal' machines.%:
% \begin{itemize}
%     \item first, the simple fact that a diagram of 2-cells as above exists forces $i$ to be an endo-1-cell $A\to A$, in order to be composable with the state 1-cell $e : A\to B$ and map to the output 1-cell $o : A\to B$; this hints at some hidden structure that the one\hyp{}object case, where every 1-cell is an endomorphism for trivial reasons, did not allow to see (cf. \autoref{must_be_endo}). The take is that input and output are not interchangeable concepts: an input 1-cell must be an endomorphism.
%     \item Second, another piece of structure that in the monoidal case remained hidden is a natural notion of 1- and 2-cell between bicategorical machines, which specialises, in the monoidal case, to a novel notion of 1- \emph{and even 2-cell} between automata; we call such arrows \emph{intertwiners} and \emph{intertwiner 2-cells}, cf. \autoref{subsec:inter}; they have apparently never been considered before, in the monoidal case.
% \end{itemize}
% \medskip
\paragraph{Related work.} A word on related work and how we fit into it: the ideas in \autoref{sec:bicats} borrow heavily from \cite{ITA_2002__36_2_181_0,rosebrugh_sabadini_walters_1998} where bicategories of automata (or `processes') are studied in fine detail; in \autoref{sec:bicats} we carry on a comparison with a different approach to bicategories of automata, present in \cite{guitart1974remarques} but also in \cite{guitart1978bimodules,ea309886425a9829ea5df96a513b157fd9940689}; in particular, our proof that there is an adjunction between the two bicategories is novel --to the best of our knowledge-- and it hints at the fact that the two approaches are far from being independent. At the level of an informal remark, the idea of approaching automata via (spans where one leg is a) fibrations bears some resemblance to Walters' work on context\hyp{}free languages through displayed categories in \cite{Walters1989}, and the requirement to have a fibration as one leg of the span should be thought as mirroring \emph{determinism} of the involved automata: if $\langle s,d\rangle : E\times M\to N\times E$ is fugal and $s$ defines a \emph{fibration} over $N$, then $E$ is a $M$-$N$-bimodule, not only an $M$-set; there is extensive work of Betti\hyp{}Kasangian \cite{bettiproperty,betti1981quasi,kasangian1983cofibrations} and Kasangian\hyp{}Rosebrugh \cite{kasangian1990glueing} on `profunctorial' models for automata, their behaviour, and the universal property enjoyed by their minimisation: spans of two\hyp{}sided fibrations \cite{StreetFibreYoneda1974,street1980fibrations} and profunctors are well\hyp{}known to be equivalent ways to present the same bicategory of two\hyp{}sided fibrations. Carrying on our study will surely determine a connection between the two approaches.

For what concerns \autoref{sec:bicat_mach}, the idea of valuing a Mealy or a Moore machine in a bicategory seems to be novel, although in light of \cite{rosebrugh_sabadini_walters_1998} and in particular of their concrete description of $\clC=\Omega\Sigma(\clK,\otimes)$ it seems that both $\Mly_\bbB$ and $\Mre_\bbB$ allow defining tautological functors into $\clC$. How these two bicategories relate is a problem we leave for future investigation: \cite{rosebrugh_sabadini_walters_1998} proves that when $\clK$ is Cartesian monoidal, $\Mly_\clK$ is $\Omega\Sigma(\clK,\times)$. The conjecture is that our $\Mly_\bbB$ is $\Omega\bbB$ under some assumptions on the bicategory $\bbB$: our notion of intertwiner seems to hint in that direction. Characterising `behaviour as a Kan extension' is nothing but taking seriously the claim that animates applications of coalgebra theory \cite{Jacobs2015,jacobs_2016} to automata; the --apparently almost unknown-- work of Bainbridge \cite{Bainbridge1975addressed} bears some resemblance to our idea, but his note is merely sketched, no plausibility for his intuition is given. Nevertheless, we recognise the potential of his idea and took it to its natural continuation with modern tools of 2-dimensional algebra.
\subsection{Mealy and Moore automata}
The scope of the following subsection is to introduce the main characters studied in the paper:\footnote{An almost identical introductory short section appears in \cite{noi:completeness}, of which the present note is a parallel submission --although related, the two manuscripts are essentially independent, and the purpose of this repetition is the desire for self\hyp{}containment.} categories of automata valued in a monoidal category $(\clK,\otimes)$ (in two flavours: `Mealy' machines, where one considers spans $E\leftarrow E\otimes I\to O$, and `Moore', where instead one consider pairs $E\leftarrow E\otimes I,E\to O$.

The only purpose of this short section is to fix the notation for \autoref{sec:bicats} and \ref{sec:bicat_mach}; comprehensive classical references for this material are \cite{adam-trnk:automata,Ehrig}.

\medskip
For the entire subsection, we fix a monoidal category $(\clK,\otimes,1)$.
\begin{definition}[Mealy machine]\label{mmach_1cells}
  A \emph{Mealy machine} in $\clK$ of input object $I$ and output object $O$ consists of a triple $(E,d,s)$ where $E$ is an object of $\clK$ and $d,s$ are morphisms in a span $\fke := \left(\mealy{E}{d}{s}\right)$.
  % \[\fke := \left(\mealy{E}{d}{s}\right)\label{mmach_eq}\]
\end{definition}
\begin{remark}[The category of Mealy machines]\label{mmach_2cells}
  Mealy machines of fixed input and output $I,O$ form a category, if we define a \emph{morphism of Mealy machines} $f : \fke=(E,d,s)\to (F, d', s')=\fkf$ as a morphism $f : E\to F$ in $\clK$ such that
  \begin{itemize}
    \item $d'\circ (f\otimes I) = f \circ d$;
    \item $s'\circ (f\otimes I) = s$.
  \end{itemize}
  % \[\vxy{
  %   E\ar[d]_f & \ar[l]_d E\otimes I \ar[d]^{f\otimes I}\ar[r]^s& O\ar@{=}[d]\\
  %   F & \ar[l]_-{d'} F\otimes I \ar[r]_{s'}& O
  %   }
  % \]
  Composition and identities are performed in $\clK$.
\end{remark}
The category of Mealy machines of input and output $I,O$ is denoted as $\Mly_\clK(I,O)$.
\begin{definition}[Moore machine]\label{moore_1cells}
  A \emph{Moore machine} in $\clK$ of input object $I$ and output object $O$ is a diagram $\fkm := \left(\moore{E}{d}{s}\right)$.
  % \[\fkm := \left(\moore{E}{d}{s}\right)\label{momach_eq}\]
\end{definition}
\begin{remark}[The category of Moore machines]\label{moore_2cells}
  Moore machines of fixed input and output $I,O$ form a category, if we define a \emph{morphism of Moore machines} $f : \fke=(E,d,s)\to (F, d', s')=\fkf$ as a morphism $f : E\to F$ in $\clK$ such that
  \begin{itemize}
    \item $d'\circ (f\otimes I) = f \circ d$;
    \item $s'\circ f = s$.
  \end{itemize}
  % \[\vxy{
  %   E\ar[d]_f & \ar[l]_d E\otimes I \ar[d]^{f\otimes I}&E\ar[d]_f \ar[r]^s& O\ar@{=}[d]\\
  %   F & \ar[l]_-{d'} F\otimes I &F\ar[r]_{s'}& O
  %   }
  % \]
\end{remark}
\begin{remark}[Canonical extension of a machine]\label{can_ext}
    If $(\clK,\otimes)$ has countable coproducts preserved by each $A\otimes \firstblank$ then the span \autoref{mmach_1cells}, considering for example Mealy machines, can be `extended' to a span
    \[\mealy{E}{d^*}{s^*}[I^*][O]\label{mly_extension}\]
    where $d^*,s^*$ can be defined inductively from components $d_n,s_n : E\otimes I^{\otimes n} \to E,O$; if $\clK$ is closed, the map $d^*$ corresponds, under the monoidal closed adjunction, to the monoid homomorphism $I^* \to [E,E]$ induced by the universal property of $I^* = \sum_{n\ge 0} I^{\otimes n}$.
\end{remark}

\section{Bicategories of automata}\label{sec:bicats}
Let $(\clK,\times)$ be a Cartesian category. There is a bicategory $\twoMly_\clK$ defined as follows (cf. \cite{rosebrugh_sabadini_walters_1998} where this is called `$\cate{Circ}$' and studied more generally, in case the base category has a non\hyp{}Cartesian monoidal structure):
\begin{definition}[The bicategory $\twoMly_\clK$, \cite{rosebrugh_sabadini_walters_1998}]\label{saba_mly}
  The bicategory $\twoMly_\clK$ has
  \begin{itemize}
    \item its \emph{0-cells} $I,O,U,\dots$ are the same objects of $\clK$;
    \item its \emph{1-cells} $I\to O$ are the Mealy machines $(E,d,s)$, i.e. the objects of the category $\Mly_\clK(I,O)$ in \autoref{mmach_2cells}, thought as morphisms $\langle s,d\rangle : E\times I \to O\times E$ in $\clK$;
    \item its \emph{2-cells} are Mealy machine morphisms as in \autoref{mmach_2cells};
    \item the composition of 1-cells $\firstblank \c \firstblank$ is defined as follows: given 1-cells $\langle s,d\rangle : E\times I\to J\times E$ and $\langle s',d'\rangle : F\times J \to K\times F$ their composition is the 1-cell $\langle s'\c s,d'\c d\rangle : (F\times E)\times I \to K\times (F\times E)$, obtained as
          \[\label{qui}\vxy{F\times E \times I \ar[r]^-{F\times \langle s,d\rangle} &
            F\times J \times E \ar[r]^-{\langle s',d'\rangle \times E} &
            K\times F\times E;}\]
    \item the \emph{vertical} composition of 2-cells is the composition of Mealy machine morphisms $f : E \to F$ as in \autoref{mmach_2cells};%, happening in $\clK$;
    \item the \emph{horizontal} composition of 2-cells is the operation defined thanks to bifunctoriality of $\firstblank\c\firstblank : \twoMly_\clK(B,C)\times \twoMly_\clK(A,B)\to \twoMly_\clK(A,C)$;
    \item the associator and the unitors are inherited from the monoidal structure of $\clK$.
  \end{itemize}
\end{definition}
\begin{remark}
  Spelled out explicitly, the composition of 1-cells in \autoref{qui} corresponds to the following morphisms (where we freely employ $\lambda$-notation available in any Cartesian closed category):
  \[
  d_2\c d_1 : \lambda efa.\langle d_2(f,s_1(e,a)),d_1(e,a)\rangle \qquad\qquad s_2 \c s_1 : \lambda efa.s_2(f,s_1(e,a))
  \label{d2d1_term}
  \]
%   It is also convenient to express the composition of 2-cells in string diagrammatic terms: the maps $d_i,s_i$ composing the Mealy machines are then the arrows
%   \[
%     \begin{tikzpicture}
%       \arTwo[green!10]{d_1} \step[1.25]{\twoAr[yellow!30]{s_1}}
%       \step[6]{\arTwo[green!10]{d_2}} \step[7.25]{\twoAr[yellow!30]{s_2}}
%     \end{tikzpicture}
%     \]
%     and their composition results respectively in the diagrams
%     \[
%       \begin{tikzpicture}
%         \twoAr[yellow!30]{s_1}
%         \down{\Did\step{\twoAr[yellow!30]{s_2}}}
%         \begin{scope}[xshift=4cm,xscale=.75]
%           \down\Did\comult\up[3]\comult
%           \step{
%             \xScale[2]{\down\Did}
%             \Did
%             \up[2.5]{\yScale[.5]\braid}
%             \up[3]\Uid
%           }
%           \step[2]{
%             \twoAr[yellow!30]{s_1}
%             \up[3]{\twoAr[green!10]{d_1}}
%           }
%           \step[3]{
%             \down{\twoAr[green!10]{d_2}}
%             \up[2]\Uid
%           }
%         \end{scope}
%       \end{tikzpicture}
%       \]
%       From these string-diagrammatic expressions, the proof that $\c$ is associative up to isomorphism follows almost immediately.
\end{remark}
% %===
\begin{remark}[Kleisli extension of automata as base changes]
If $P : \clK \to\clK$ is a commutative monad \cite{kock:cc-commnd,kock1972strong}, we can lift the monoidal structure $(\clK,\otimes)$ to a monoidal structure $(\Kl(P), \bar\otimes)$ on the Kleisli category of $P$; this leads to the notion of \emph{$P$-non\hyp{}deterministic automata} or \emph{$P_\lambda$-machines} studied in \cite[§2, Définition 6]{Guitart1980}. Nondeterminism through the passage to a Kleisli category is a potent idea that developed into the line of research on automata theory through coalgebra theory \cite{jacobs_2016}, cf. in particular Chapter 2.3 for a comprehensive reference, or the self\hyp{}contained \cite{Jacobs2006}.
\end{remark}
We do not investigate the theory of $P_\lambda$-machines apart from the following two results the proof of which is completely straightforward: we content ourselves with observing that the results expounded in \cite{Katis1997,rosebrugh_sabadini_walters_1998}, and in general the language of bicategories of processes, naturally lends itself to the generation of \emph{base\hyp{}change functors}, of which the following two are particular examples.% after we prove \autoref{chase_bange}.
\begin{proposition}\label{K_to_KL}
  The correspondence defined at the level of objects by sending $(E,d,s)\in\Mly_{\clK}(I,O)$ to
  \[\xymatrix@C=5mm{
    PE & \ar[l]_{\eta_E} E &\ar[l]_d  E\otimes I \ar[r]^s & O \ar[r]^{\eta_O}& PO
    }\]
  extends to a functor $L : \Mly_\clK(I,O)\to \Mly_{\Kl(P)}(I,O)$.
\end{proposition}
\begin{proposition}\label{KL_to_K}
  The correspondence sending $(E,d,s)\in\Mly_{\Kl(P)}(I,O)$ into
  \[\xymatrix{
    PE &\ar[l]_-{\mu_E} PPE &\ar[l]_-{Pd\circ D} PE\otimes PI \ar[r]^-{Ps\circ D} & PPO \ar[r]^-{\mu_O}& PO
    }\]
  extends to a functor $(-)^\text{e} : \Mly_{\Kl(P)}(I,O)\to \Mly_\clK(PI,PO)$.
\end{proposition}
More precisely, the proof of the following result is straightforward --only slightly convoluted in terms of notational burden-- so much so that we feel content to enclose it in a remark.
\begin{remark}
    Let $\clH,\clK$ be cartesian monoidal categories, then we can define 2-categories $\twoMly_\clH,\twoMly_\clK$ as in \autoref{saba_mly}; let $F : \clH\to\clK$ be a lax monoidal functor. Then, there exists a `base change' pseudofunctor $F_* : \twoMly_\clH \to \twoMly_\clK$, which is the 1-cell part of a 2-functor $\Cat_\times\to\cate{Bicat}$ defined on objects as $\clK\mapsto\Mly_\clK$, from (Cartesian monoidal categories, product\hyp{}preserving functors, Cartesian natural transformations), to (bicategories, pseudofunctors, oplax natural transformations).
\end{remark}
As a corollary, we re\hyp{}obtain the functors of \autoref{KL_to_K} and \autoref{K_to_KL} from the free and forgetful functors $F_P : \clK\to \Kl(P)$ and $U_P : \Kl(P)\to\clK$.
%===
\subsection{Fugal automata, Guitart machines}
A conceptual construction for $\twoMly_\clK$ in \autoref{saba_mly} is given as follows in \cite{Katis1997}: it is the category $\Omega\Sigma(\clK,\otimes)$ of pseudofunctors $\bbN \to \Sigma(\clK,\times)$ and lax transformations, where $\Sigma$ is the `suspension' of $(\clK,\otimes)$, i.e. $\clK$ regarded as a one\hyp{}object bicategory; a universal property for $\twoMly_\clK$ is provided in \cite{ITA_2002__36_2_181_0} (actually, for any $\Omega\Sigma(\clK,\otimes)$): it is the free \emph{category with feedbacks} (op. cit., Proposition 2.6, see also \cite{openTransitionSystems21}) on $\clK$. The bicategory $\twoMly_\clK$ addresses the fundamental question of whether one can fruitfully consider morphisms in a category as an abstraction of `sequential operations' performed on a domain/input to obtain a codomain/output, and up to what point the analogy between morphisms and sequential machines holds up (composing 1-cells in $\Mly_\clK$ accounts for the sequential composition of state machines, where the state $E$ is an intrinsic part of the specification of a machine/1-cell $\langle s,d\rangle$).

Twenty eight years before \cite{ITA_2002__36_2_181_0}, however, René Guitart \cite{guitart1974remarques} exhibited another bicategory $\preMac$ of `Mealy machines', defined as a suitable category of spans, of which one leg is a fibration, and its universal property: $\preMac$ is the Kleisli bicategory of the diagram monad (\emph{monade des diagrammes} in \cite[§1]{guitart1974remarques}, cf. \cite{kock1967limit,Perrone2022}) $\Cat/\!\!/\firstblank$.\footnote{Guitart's note \cite{guitart1974remarques} is rather obscure with respect of the fine details of his definition, as he chooses for 2-cells the $H$ for which the upper triangle in \eqref{qua} is only \emph{laxly} commutative, and when it comes to composition of 1-cells he invokes a \emph{produit fibré canonique}; apparently, this can't be interpreted as a strict pullback, or there would be no way to define horizontal composition of 2-cells; using a comma object instead of a strict pullback, the lax structure is given by the universal property --observe that the functor that must be an opfibration is indeed an opfibration, thanks to \cite[Exercise 1.4.6]{CLTT}, but this opfibration does not remember much of the opfibration $q$ one pulled back. Our theorem involves a strict version of Guitart's $\preMac$, because the functor $\Pi$ of \autoref{fun_mly_mac} factors through $\Mac\subseteq\preMac$.}
% \end{itemize}
\begin{definition}[The bicategory $\Mac$, adapting \cite{guitart1974remarques}]\label{guita_mly}
  Define a bicategory $\Mac$ as follows:
  \begin{itemize}
    \item 0-cells are categories $\clA,\clB,\clC$\dots;
    \item 1-cells $(\clE;p,S) : \clA\to\clB$ consist of spans
          \[\vxy{\clA &\ar@{->>}[l]_-p \clE\ar[r]^-S & \clB}\]
          where $p : \clE\to\clA$ is a discrete opfibration;
    \item 2-cells $H : (\clE;p,S)\To (\clF;q,T)$ are pairs where $H : \clE\to \clF$ is a morphism of opfibrations (cf. \cite[dual of 1.7.3.(i)]{CLTT}): depicted graphically, a 2-cell is a diagram
          \[\label{qua}\vxy{
              \clE\ar[dr]_-H\ar[r]^S\ar@{->>}[d]_p  &\clB \\
              \clA  & \clF\ar@{->>}[l]^q \ar[u]_T
            }\]
          where both triangles commute and $H$ is an opCartesian functor (it preserves opCartesian morphisms);
    \item composition of 1-cells $\clA \xot p \clE\xto S \clB$ and $\clB \xot q \clF \xto T \clC$ is via pullbacks, as it happens in spans, and all the rest of the structure is defined as in spans.
  \end{itemize}
\end{definition}
Given this, a natural question that might arise is how do the two bicategories of \autoref{saba_mly} and \autoref{guita_mly} interact, if at all?

In the present section, we aim to prove the existence of an adjunction (cf. \autoref{the_adjunzia}) between a suitable sub\hyp{}bicategory of $\Mac$ and a sub\hyp{}bicategory of $\Mly_\Set$ spanned over what we call \emph{fugal} Mealy machines between monoids (cf. \autoref{def:beh_mealy}).\footnote{A \emph{fugue} is `a musical composition in which one or two themes are repeated or imitated by successively entering voices and contrapuntally developed in a continuous interweaving of the voice parts', cf. \cite{merriam_webster}. In our case, the interweaving is between $s,d$ in a Mealy machine.}

Since the construction of $\Mac$ outlined in \cite{guitart1974remarques} requires some intermediate steps (and it is written in French), we deem it necessary to delve into the details of how its structure is presented. To fix ideas, we keep working in the category of sets and functions.
\begin{notation}
  In order to avoid notational clutter, we will blur the distinction between a monoid $M$ and the one\hyp{}object category it represents; also, given the $d$ part of a Mealy machine, we will denote as $d^*$ both the extension $E\times I^* \to E$ of \autoref{can_ext}, which is a monoid action of $I^*$ on $E$, and the functor $I^* \to\Set$ to which the action corresponds.
\end{notation}
\begin{remark}\label{cat_elem}
  In the notation above, a Mealy machine $\fke=(E,d,s)$ yields a discrete opfibration (cf. \cite{acc,CLTT}) $\clE[d^*] \to I^*$ over the monoid $I^*$, and $\clE[a]$ is the \emph{translation category} of an $M$-set $a : M\times X\to X$ (cf. \cite{ronnie} for the case when $M$ is a group: clearly, $\clE[a]$ is the category of elements of the action $a : M\to\Set$ regarded as a functor), i.e. the category having
  \begin{itemize}
    \item objects the elements of $E$;
    \item a morphism $m : e\to e'$ whenever $e'=d^*(e,m)$.
  \end{itemize}
  Composition and identities are induced by the fact that $d^*$ is an action.
\end{remark}
\begin{remark}\label{up_4_MAC}
  The hom\hyp{}categories $\Mac(\clA,\clB)$ of \autoref{guita_mly} fit into strict pullbacks
  \[\vxy{
      \Mac(\clA,\clB) \xpb \ar[r]\ar[d]& \Cat\ls \clB \ar[d]\\
      \opFib/\clA \ar[r] & \Cat
    }\]
  where $\Cat\ls \clB$ is the usual slice category of $\Cat$ over $\clB$.
\end{remark}
\begin{definition}[Fugal automaton]\label{def:beh_mealy}
  Let $M,N$ be monoids; a Mealy machine $\langle s,d\rangle : E\times M \to N \times E$ is \emph{fugal} if its $s$ part satisfies the equation
  \[s(e,m\cdot m') = s(e,m)\cdot s(d(e,m),m').\label{beheq}\]
\end{definition}

\begin{remark}\label{why_fugal}
  This definition appears in \cite[§2]{guitart1974remarques} and it looks  an ad\hyp{}hoc restriction for what an output map in a Mealy machine shall be; but \eqref{beheq} can be motivated in two ways:
  \begin{itemize}
    \item A fugal Mealy machine $\langle s ,d\rangle : E\times M \to N \times E$ induces in a natural way a functor $\Sigma : \clE[d^*] \to N$ because \eqref{beheq} is exactly equivalent to the fact that $\Sigma$ defined on objects in the only possible way, and on morphisms as $\Sigma(e\to d^*(e,m)) = s(e,m)$ preserves (identities and) composition;
    \item given a generic Mealy machine $\langle s,d\rangle : E\times A \to B \times E$ one can produce a `universal' fugal Mealy machine $\langle s,d \rangle^\flat =\langle s^\flat,\firstblank\rangle : E\times A^* \to B^* \times E$, and this construction is well\hyp{}behaved for 1-cell composition in $\Mly_\Set$, in the sense that $(s_2\c s_1)^\flat = s_2^\flat\c s_1^\flat$.
  \end{itemize}
\end{remark}
The remainder of this section is devoted to making these claims precise (and prove them). In particular, the `universality' of $\langle s,d\rangle^\flat$ among fugal Mealy machines obtained from $\langle s,d\rangle$ is clarified by the following \autoref{lemma_1} and by \autoref{the_adjunzia}, where we prove that there is a 2-adjunction between $\Mly_\Set$ and $\bMly_\Set$.
\begin{lemma}\label{lemma_1}
  Given sets $A,B$, denote with $A^*,B^*$ their free monoids; then, there exists a `fugal extension' functor $(\firstblank)_{A,B}^\flat : \Mly_\Set(A,B) \to \bMly_\Set(A^*,B^*)$.
\end{lemma}
\begin{proof}
  \deferredRef{lemma_1}. In particular, the map $s^\flat$ is constructed inductively as
  \[\begin{cases}
      s^\flat(e, [\,])      & = [\,]                            \\
      s^\flat(e, a\kons as) & = s(e,a) \kons s^\flat(d(e,a),as)
    \end{cases}\]
  and it fits in the Mealy machine $\langle s^\flat,d^*\rangle : E\times A^* \to B^*\times E$ where $d^*$ is as in \eqref{mly_extension}. The proof that $\langle s^\flat,d^*\rangle$ is fugal in the sense of \eqref{beheq} can be done by induction and poses no particular difficulty.
\end{proof}
\begin{lemma}\label{lemma_2}
  Given sets $A,B$ there exists a commutative square
  \[
    \vxy{
      \bMly_\Set(A^*,B^*)\ar[r]\ar[d]& \Cat \ls B^* \ar[d]\\
      \opFib/A^* \ar[r]& \Cat.
    }
  \]
\end{lemma}
\begin{proof}[Proof of \autoref{lemma_2}]\label{proof_of_lemma_2}
  Given a fugal Mealy machine $\langle s,d\rangle : E \times A^* \to B^*\times E$ between free monoids, from the action $d$ we obtain a discrete opfibration $\clE[d]\to A^*$, and from the map $s : E \times A^* \to B^*$ we obtain a functor $\Sigma : \clE[d^*]\to B^*$ as in \autoref{why_fugal}. So, one can obtain a span
  \[\vxy{A^* & \ar@{->>}[l]_D \clE[d^*] \ar[r]^\Sigma & B^*}\label{eq:proof_lemma_2}\]
  where the leg $D : \clE[d^*] \to A^*$ is as in \autoref{cat_elem} and $\Sigma$ is an in \autoref{why_fugal}. The functors $\opFib/A^*\leftarrow \bMly_\Set(A^*,B^*)\to \Cat\ls B^*$ project to each of the two legs.
\end{proof}
\begin{corollary}\label{the_maps}
  The universal property of the hom\hyp{}categories $\Mac(\clA,\clB)$ exposed in \autoref{up_4_MAC} yields the right\hyp{}most functor in the composition
  \[\vxy[@C=1.5cm]{
    \Gamma_{A,B} : \Mly_\Set(A,B) \ar[r]^-{(\firstblank)_{A,B}^\flat} & \bMly_\Set(A^*,B^*) \ar[r]^-{\Pi_{A,B}} & \Mac(A^*,B^*)
    }\]
\end{corollary}
\begin{lemma}[Fugal extension preserves composition]\label{is_beh}
  Let $A,B,C$ be sets, $s_1 : E \times A \to B$ and $s_2 : F\times B\to C$ parts of Mealy machines $\langle s_1,\firstblank\rangle$ and $\langle s_2,\firstblank\rangle$; then $(s_2\c s_1)^\flat = s_2^\flat\c s_1^\flat$.
  % \[(s_2\c s_1)^\flat = s_2^\flat\c s_1^\flat.\]
\end{lemma}
\begin{proof}
  \deferredRef{is_beh}.\footnote{The argument is straightforward but tedious (the difficult part is that the condition to verify on $(s_2\c s_1)^\flat$ involves $d_2\c d_1$, the expression of which we recall from \eqref{d2d1_term}, is the $\lambda$-term $\lambda efa.\langle d_2(f,s_1(e,a)),d_1(e,a)\rangle$).}
\end{proof}
This, together with the fact that the identity 1-cell $1\times A\to A\times 1$ is fugal (the proof is straightforward), yields that there exists a 2-subcategory $\bMly_\Set$ of $\Mly_\Set$ where 0-cells are monoids, 1-cells are the $\langle s,d\rangle$ where $s$ is fugal in the sense of \autoref{def:beh_mealy}, and we take all 2-cells.
\begin{theorem}\label{fun_mly_mac}
  The maps $\Gamma_{A,B}$ of \autoref{the_maps} constitute the action on 1-cells of a 2-functor $\Gamma : \twoMly_\Set \to \Mac$. More precisely, there are 2-functors $(\firstblank)^\flat : \twoMly_\Set \to \bMly_\Set$ and $\Pi : \bMly_\Set\to\Mac$ whose composition is $\Gamma$.
\end{theorem}
\begin{proof}
  \deferredRef{fun_mly_mac}.
\end{proof}
\begin{theorem}\label{the_adjunzia}
  The 2-functor $(\firstblank)^\flat : \twoMly_\Set \to \bMly_\Set$ admits a right 2-adjoint; the 2-functor $\Pi : \bMly_\Set\to\Mac$ identifies $\bMly_\Set$ as the 1-full and 2-full subcategory of $\Mac$ spanned by monoids.
\end{theorem}
\begin{proof}
  \deferredRef{the_adjunzia}. The last statement essentially follows from \eqref{eq:proof_lemma_2}: the span $(D,\Sigma)$ is essentially equivalent to the fugal Mealy machine $\langle s,d\rangle$, since its left leg $D$ determines a unique action of $A^*$ on the set of objects $\clE[d^*]_0$, and $\Sigma$ and $s$ are mutually defined.
\end{proof}

\section{Bicategory-valued machines}\label{sec:bicat_mach}
A monoidal category is just a bicategory with a single 0-cell; then, do \autoref{mmach_1cells} and \autoref{moore_1cells} admit a generalisation when instead of $\clK$ we consider a bicategory $\bbB$ with more than one object? The present section answers in the positive. We also outline how, passing to automata valued in a bicategory, a seemingly undiscovered way to define morphisms between automata, different (from \eqref{mmach_2cells} and) from the categories of `variable' automata described in \cite[§11.1]{Ehrig}: we study this notion in \autoref{new_mor}.

In our setting, `automata' become diagrams of \emph{2-cells} in $\bbB$, between input, output and state \emph{1-cells}, in contrast with previous studies where automata appeared as objects, and with \cite{rosebrugh_sabadini_walters_1998} (and our \autoref{sec:bicats}), where they appear as diagrams of 1-cells between input, output and state 0-cells. This perspective suggests that 2-dimensional diagrams of a certain shape can be thought of as state machines -so, they carry a computational meaning; but also that state machines can be fruitfully interpreted as diagrams: in \autoref{in_rel} we explore definitions of an automaton where input and output are \emph{relations}, or functors (in \autoref{in_cat}), or profunctors (in \autoref{in_prof}); universal objects that can be attached to the 2-dimensional diagram then admit a computational interpretation (cf. \eqref{reach_eq} where a certain Kan extension resembles a `reachability' relation).

This idea is not entirely new: it resembles an approach contained in \cite{Bainbridge1975addressed,bainbridge1972unified} where the author models the state space of abstract machines as a functor, of which one can take the left/right Kan extension along an `input scheme'. However, Bainbridge's works are rather obscure (and quite ahead of their time), so we believe we provide some advancement to state of the art by taking his idea seriously and carrying to its natural development --while at the same time, providing concrete examples of bicategories in which inputs/outputs automata can be thought of as 1-cells, and investigating the structure of the class of all such automata as a global object.
\begin{definition}
  Adapting from \autoref{mmach_1cells} \emph{verbatim}, if $\bbB$ is a bicategory with 0-cells $A,B,X,Y,\dots$, 1-cells $i : A\to B, o : X\to Y,\dots$ and 2-cells $\alpha,\beta,\dots$ the kind of object we want in $\Mly_\bbB(i,o)$ is a span of the following form:
  \[\vxy{e & \ar@{=>}[l]_\delta e\circ i \ar@{=>}[r]^\sigma & o}\label{bica_mealy}\]
  for 1-cells $i : X\to Y$, $e : A\to B$, $o : C\to D$. Note that with $\agdacirc$, we denote the composition of 1-cells in $\bbB$, which becomes a monoidal product in $\bbB$ has a single 0-cell.
\end{definition}
\begin{remark}\label{must_be_endo}
  The important observation here is that the mere existence of the span $(\delta,\sigma)$ `forces the types' of $i,o,e$ in such a way that $i$ \emph{must} be an endomorphism of an object $A\in\bbB$, and $e,o : A\to B$ are 1-cells. Interestingly, these minimal assumptions required even to consider an object like \eqref{bica_mealy} make iterated compositions $i\circ \dots\circ i$ as meaningful as iterated tensors $I\otimes \dots\otimes I$, and in fact, the two concepts coincide when $\bbB$ has a single object $*$ and hom\hyp{}category $\bbB(*,*)=\clK$.

  In the monoidal case, the fact that an input 1-cell stands on a different level from an output was completely obscured by the fact that \emph{every} 1-cell is an endomorphism.
\end{remark}
Let us turn this discussion into a precise definition.
\begin{definition}[Bicategory-valued Mealy machines]\label{def:bicatmealy}
  Let $\bbB$ be a bicategory, and fix two 1-cells $i : A\to A$ and $o : A\to B$; define a category $\Mly_\bbB(i,o)$ as follows:
  \begin{enumtag}{bml}
    \item the objects are diagrams of 2-cells as in \eqref{bica_mealy};
    % \[\label{objects_of_bicatmealy}
    %   \vxy{
    %     A
    %     \ar[d]_i
    %     \dduppertwocell^{o}{^<-2>\sigma}
    %     \ddlowertwocell_e{<2>\delta}\\
    %     A \ar[d]_e\\
    %     B
    %   }\]
    \item the morphisms $(e,\delta,\sigma)\to (e',\delta',\sigma')$ are 2-cells $\varphi : e\To e'$ subject to conditions similar to \autoref{mmach_2cells}:%, which when expanded into 2-dimensional pasting diagrams become the following identities:
    \begin{itemize}
      \item $\sigma'\circ (\varphi * i) = \sigma$;
      \item $\delta'\circ (\varphi * i) = \varphi \circ\delta$.
    \end{itemize}
    % \[\vxy{
    %   A\ar@/_2pc/[dd]_{e'} \ddlowertwocell<\omit>{<3>\delta'}\ar@{->}[d]^i & &  & A \ar[d]^i \ar@/_3pc/[dd]_{e'}\ddlowertwocell_{}{\delta}\ddtwocell<\omit>{<5.25>\varphi} &&& A\ar@/_2pc/[dd]_o \ddlowertwocell<\omit>{<3>\sigma'}\ar@{->}[d]^i & & A \ar[d]^i \ddlowertwocell_{}{\sigma} \\
    %   A\dtwocell^e_{e'}{\varphi} & = & & A\ar[d]^e &;&& A\dtwocell^e_{e'}{\varphi} & = & A\ar[d]^e  \\
    %   B &  & & B &&& B &  & B
    %   }\label{bicatmealy_2cell}\]
    % are commutative.
  \end{enumtag}
\end{definition}
\begin{definition}[Bicategory-valued Moore machines]\label{def:bicatmoore}
  Define a category $\Mre_\bbB(i,o)$ as follows:
  \begin{enumtag}{bmo}
    \item\label{bmo_1} the objects are pairs of 2-cells in $\bbB$, $\delta : e\circ i \To e$ and $\sigma : e\To o$;
    % \[\vxy{e & \ar@{=>}[l]_\delta e\circ i & e \ar@{=>}[r]^\sigma & o}\label{bmo_1}\]
    \item\label{bmo_2} the morphisms $(e,\delta,\sigma)\to (e',\delta',\sigma')$ are 2-cells $\varphi : e\To e'$ such that diagrams of 2-cells similar to those in \autoref{def:bicatmealy} are commutative.
  \end{enumtag}
\end{definition}
\begin{notation}
  In the following, an object of $\Mly_\bbB(i,o)$ (resp., $\Mre_\bbB(i,o)$) will be termed a \emph{bicategorical Mealy machine} (resp., a \emph{bicategorical Moore machine}) of input cell $i$ and output cell $o$, and the objects $A,B$ are the \emph{base} of the bicategorical Mealy machine $(e,\delta,\sigma)$. To denote that a bicategorical Mealy machine is based on $A,B$ we write $(e,\delta,\sigma)_{A,B}$.
\end{notation}
In \cite{Bainbridge1975addressed} the author models the state space of abstract machines as follows: fix categories $A,X,E$ and a functor $\Phi : X\to A$, of which one can take the left/right Kan extension along an `input scheme' $u : E\to X$; a \emph{machine with input scheme $u$} is a diagram of 2-cells in $\Cat(E,A)$ of the form $\clM = (I \To \Phi\circ u\To J)$, and the \emph{behaviour} $B(\clM)$ of $\clM$ is the diagram of 2-cells $\Lan_u I \To \Phi\To \Ran_u J$.

All this bears some resemblance to the following remark, but at the same time looks very mysterious, and not much intuition is given in \emph{op. cit.} for what the approach in study means; we believe our development starts from a similar point (the intuition that a category of machines is, in the end, some category of diagrams --a claim we substantiate in \autoref{mre_as_diag}) but rapidly takes a different turn (cf. \autoref{new_mor}), and ultimately gives a cleaner account of Bainbridge's perspective (see also \cite{bainbridge1972unified} of the same author).% can be translated without effort to be valid in a bicategory with enough 2-dimensional colimits. However, our approach starts from the idea that a category of machines is, in the end, merely a category of diagrams (something we prove in \autoref{mre_as_diag}): there is some (equally forgotten) work hinting in this direction, cf. \cite{guitart1978bimodules,ea309886425a9829ea5df96a513b157fd9940689}.
\begin{remark}[Behaviour as a Kan extension]\label{as_a_kan}
  A more convenient depiction of the span in \ref{bmo_1} will shed light on our \autoref{def:bicatmealy} and \ref{def:bicatmoore}, giving in passing a conceptual motivation for the convoluted shape of finite products in $\Mre_\clK(I,O)$ and $\Mly_\clK(I,O)$ (cf. \cite[Ch. 11]{Ehrig}): a bicategorical Moore machine in $\bbB$ of fixed input and output $i,o$ consists of a way of filling the dotted arrows in the diagram
  \[\label{mach_as_kan}\vxy{
    &A\ar[dl]_i\ar@{.>}[dr]_e\druppertwocell^o{^\sigma}&\\
    A\ar@{.>}[rr]_e&\utwocell<\omit>{\delta}&B    }\]
  with $e : A\to B$ and two 2-cells $\delta,\sigma$. But then the `terminal way' of filling such a span can be characterised by the right extension of the output object along a certain 1-cell obtained from the input $i$. Let us investigate how.

  First of all, we have to assume something on the ambient hom\hyp{}categories $\bbB(A,A)$, namely that each of these admits a left adjoint to the forgetful functor
  \[\vxy{\bbB(A,A)\ar[r] & \cate{Mnd}_{/A}}\]
  (cf. \cite[Ch. II]{dubuc}) so that every endo-1-cell $i : A\to A$ has an associated extension to an endo-1-cell $i^\natural : A\to A$ with a unit map $i\To i^\natural$ that is initial among all 2-cells out of $i$ into a monad in $\bbB$; $i^\natural$ is usually called the \emph{free monad} on $i$.
\end{remark}
\begin{construction}\label{la_ran}
  Now, fix $i,o$ as in \autoref{def:bicatmoore}; we claim that the terminal object of $\Mre_\bbB(i,o)$ is obtained as the right extension in $\bbB$ of the output $o$ along $i^\natural$. We can obtain
  \begin{itemize}
    \item from the unit $\boldsymbol{\eta} : \id_A\To i^\natural$ of the free monad on $i$, a canonical modification $\Ran_i\To\Ran_{\id} =\id_A$, with components at $o$ given by 2-cells $\sigma : \Ran_io \To o$; this is a choice of the right leg for a diagram like \ref{bmo_1};
    \item from the multiplication $\boldsymbol{\mu} : i^\natural\circ i^\natural\To i^\natural$ of the free monad on $i$, a canonical modification $\Ran_{i^\natural}\To\Ran_{i^\natural}\circ \Ran_{i^\natural}$, whose components at $o$ mate to a 2-cell $\delta_0 : \Ran_{i^\natural}o\circ i^\natural \To \Ran_{i^\natural}o$; the composite
          \[\vxy[@C=1.5cm]{\delta : \Ran_{i^\natural}o\circ i \ar@{=>}[r]^-{\Ran_{i^\natural}o * \boldsymbol{\eta}} & \Ran_{i^\natural}o\circ i^\natural \ar@{=>}[r] & \Ran_{i^\natural}o}\label{raneq}\]
          The left leg is now chosen for a diagram like \ref{bmo_1}.
  \end{itemize}
  Together, $(\Ran_{i^\natural}o,\delta,\sigma)$ is a bicategorical Mealy machine, and the universal property of the right Kan extension says it is the terminal such. A similar line of reasoning yields the same result for $\Mly_\bbB(i,o)$, only now $\sigma$ is the 2-cell obtained as mate of $\epsilon\circ (\Ran_{i^\natural}o * \boldsymbol{\eta}) : \Ran_{i^\natural}o \circ i \To \Ran_{i^\natural}o \circ i^\natural \To o$ from the counit of $\firstblank\circ i^\natural\dashv \Ran_{i^\natural}$.
  \begin{proposition}[$\Mre_\bbB(i,o)$ and $\Mly_\bbB(i,o)$ as categories of diagrams.]\label{mre_as_diag}
    There exists a 2-category $\clP$ and a pair of strict 2-functors $W,G : \clP \to \bbB$ such that bicategorical Moore machines with `variable output 1-cell' i.e. the 2-dimensional diagrams like in \eqref{mach_as_kan} where $o$ is variable, can be characterised as natural transformations $W\To G$.
  \end{proposition}
  \begin{proof}
    \deferredRef{mre_as_diag}.

    As explained therein, bicategorical Moore machines with fixed output $o$ can be characterised as particular such natural transformations that take value $o$ on one argument.

    Also, minor adjustments to the shape of $G$ yield a similar result for bicategorical Mealy machines.
  \end{proof}
  % \todo[inline]{Build the free monad on $i$; the terminal object is $\Ran_{i^*}o$.}
\end{construction}
\begin{example}[Bicategorical machines in $\Cat$]\label{in_cat}
  Consider a span $\clC \xot I \clC \xto O \clD$ in the strict 2-category $\mymathbb{Cat}$ of categories, functors and natural transformations, where $\clD$ is a $\kappa$-complete category. The category $\Mre_{\Cat}(I,O)$ has objects the triples $(E,\delta,\sigma)$ where $E : \clC\to\clD$ is a functor and $\sigma,\delta$ are natural transformations arranged as in \eqref{mach_as_kan}; assuming enough limits in $\clD$, we can compute the action of the right Kan extension of $O$ along $I^\natural$ (the free monad on the endofunctor $I$, cf. \cite{kelly1980unified}, whose existence requires additional assumptions on $\clC$) on an object $C\in\clC$ as the equaliser
  \[\vxy{
    RC \ar[r]& \prod_{C\in\clC} OC^{\clC(A,I^\natural C)}\ar@<3pt>[r]\ar@<-3pt>[r]& \prod_{C\to B}OB^{\clC(A,I^\natural C)}
    }\label{the_end}\]
  or (better, cf. \cite[2.3.6]{coend-calcu}) as the end\footnote{Recall that if $S$ is a set and $C$ is an object of a category $\clC$ with limits, by $C^S$ we denote the \emph{power} of $C$ and $S$, i.e. the iterated product $\prod_{s\in S} C$ of as many copies of $C$ as there are elements in $S$.} $A\mapsto\int_C OC^{\clC(A,I^\natural C)}$, i.e. as the `space of fixpoints' for the conjoint action of the functor $O$ and of the presheaf $C\mapsto \clC(A,I^\natural C)$ on objects of $\clC$; the free monad $I^\natural$ sends an object $C$ to the initial algebra of the functor $A\mapsto C + IA$, so that $I^\natural C\cong C+I(I^\natural C)$.% the intuitive meaning of such an object containing the `behaviour' of the system can be explained as follows:
  % Writing the coend above as
  % \[\int_C\prod_{A\to I^\natural C} OC\]

  For the sake of simplicity, let us specialise the discussion when $\clD$ is the category of sets and functions: the input $I$ and the output $O$ of the state machine in \autoref{mmach_1cells} are now variable objects `indexed' over the objects of $\clC$, and the behaviour of the terminal machine can be described as a known object: unpacking the end \eqref{the_end} we obtain the functor
  \[\vxy{A \ar@{|->}[r] & [\clC,\Set](\clC(A,I^\natural\firstblank), O)}\]
  sending an object $A$ to the set of natural transformations $\alpha :\clC(A,I^\natural\firstblank)\To O$; the intuition here is that to each generalised $A$-element of $I^\natural C$ corresponds an element of the output space $\Upsilon_C(u)\in OC$, and that this association is natural in $C$.
\end{example}
\begin{example}[Bicategorical machines in profunctors]\label{in_prof}
  We can reason similarly in the bicategory of categories and profunctors of \cite{justesen1968bikategorien,benabou2000distributors,cattani2005profunctors}, \cite[Ch. 5]{coend-calcu}; now an endo-1-cell $I : \clC\to \clC$ on a category $\clC$ consists of an `extension' of the underlying graph of $U\clC$ to a bigger graph $(U\clC)^+$,\footnote{More precisely, to the underlying graph of $\clC$, made of `old' arrows, we adjoin a directed edge $e_x : C\to C'$ for each $x\in I(C,C')$.} and the free promonad $I^\natural$ (cf. \cite[§5]{Lack2008}) corresponds to the quotient of the free category on $(U\clC)^+$ where `old' arrows compose as in $\clC$, and `new' arrows compose freely; moreover, all right extensions $\langle P/Q\rangle : \clX\pro \clY$ of $Q : \clA\pro\clY$ along $P:\clA\pro \clX$ exist in the bicategory $\mymathbb{Prof}$, as they are computed as the end in \cite[5.2.5]{coend-calcu},
  \[\vxy{\langle P/Q\rangle : (X,Y)\ar@{|->}[r] & \int_A\Set(P(Y,A),Q(X,A)).}\]
\end{example}
\begin{example}[Bicategorical machines in relations]\label{in_rel}
  When it is instantiated in the (locally thin) bicategory of relations between sets, i.e. $\{0,1\}$-profunctors, given $I : A\pro A$, $O : A\pro B$, $I^\natural$ is the reflexive\hyp{}transitive closure of $I$, and the above Kan extension is uniquely determined as the maximal $E$ such that $E\subseteq O$ and $E\circ I^\natural \subseteq E$ (here $\circ$ is the relational composition). % by the universal property $\cate{Rel}(A,B)(E\otimes I^\natural,O)\cong\cate{Rel}(A,B)(E, \Ran_{I^\natural}O)$:
  So $R=\Ran_{I^\natural}O$ is the relation defined as
  \[(a,b)\in R \iff \forall a'\in A.((a',a) \in I^\natural \To (a',b)\in O).\]
  This relation expresses \emph{reachability} of $b$ from $a$: it characterises the sub\hyp{}relation of $O$ connecting those pairs $(a,b)$ for which, for every other $a'\in A$, if there is a finite path (possibly of length zero, i.e. $a=a'$) connecting $a',a$ through $I$-related elements, then $(a',b)\in O$. In pictures:
  \[\label{reach_eq}
    a \, R \, b \iff
    \Big(
    (a'=a) \lor (a' \xto I a_1 \xto I \dots \xto I a_n \xto I a)
    \To
    a'\, O \, b
    \Big)
  \]
  When the above example is specialised to the case when $A=*$ is a singleton, there are only two possible choices for $I$ (both reflexive and transitive), and $O$ identifies to a subset of $B$; a bicategorical Moore machine is then a subset $R\subseteq O$, and thus for both choices of $I$, $\Mre_{\Rel}(I,O)_{*,B} = 2^O$. One can reason in the same fashion for Mealy machines.
\end{example}
\subsection{Intertwiners between bicategorical machines}\label{subsec:inter}
In passing from $\Mly_\clK(I,O)$ to $\Mly_\bbB(i,o)$ we gain an additional degree of freedom by being able to index the category over pairs of 0-cells of $\bbB$, and this is particularly true in the sense that the definition of $\Mly_\bbB(i,o)$ and its indexing over pairs of objects $A,B$ of $\clK$ leads to a seemingly undiscovered way to define morphisms between automata:%, different (from \eqref{mmach_2cells} and) from the categories of `variable' automata described in \cite[§11.1]{Ehrig}.
\begin{definition}[Intertwiner between bicategorical machines]\label{new_mor}
  Consider two bicategorical Mealy machines $(e,\delta,\sigma)_{A,B}, (e',\delta',\sigma')_{A',B'}$ on different bases (so in particular $(e,\delta,\sigma)_{A,B}\in\Mly_\bbB(i,o)$ and $(e',\delta',\sigma')_{A',B'}\in\Mly_\bbB(i',o')$); an \emph{intertwiner} $(u,v) : (e,\delta,\sigma) \looparrowright (e',\delta',\sigma')$ consists of a pair of 1-cells $u : A\to A', v : B\to B'$ and a triple of 2-cells $\iota,\epsilon,\omega$ disposed as in \eqref{cellette},
  to which we require to satisfy the identities in \eqref{identizie} (we provide a `birdseye' view of the commutativities that we require, as \eqref{cellette} is unambiguous about how the 2-cells $\iota,\delta,\sigma,\epsilon,\omega$ can be composed).
\end{definition}
\begin{remark}\label{not_triv}
  Interestingly enough, when it is spelled out in the case when $\bbB$ has a single 0-cell, this notion does not reduce to \autoref{mmach_2cells}, as an intertwiner between a Mealy machine $(E,d,s)_{I,O}$ and another $(E',d',s')_{I',O'}$ consists of a pair of objects $U,V\in\clK$, such that
  \begin{enumtag}{ic}
    \item there exist morphisms
    $\iota : I'\otimes U\to V\otimes I, \epsilon : E'\otimes U\to V\otimes E, \omega : O'\otimes U\to V\otimes O$;
    \item the following two identities hold:
    \begin{gather*}
      \epsilon\circ (d'\otimes U) = (V\otimes d)\circ (\epsilon\otimes I)\circ (E'\otimes\iota)\\
      \omega \circ (s'\otimes U) = (V\otimes s)\circ (\epsilon\otimes I)\circ (E'\otimes\iota)
    \end{gather*}
  \end{enumtag}
\end{remark}
In the single\hyp{}object case, this notion does not trivialise in any obvious way, and --in stark contrast with the notion of morphism of automata given in \eqref{mmach_2cells}-- intertwiners between machines support a notion of higher morphisms \emph{even in the monoidal case}.
\begin{definition}[2-cell between machines]\label{anew_twocel}
  In the same notation of \autoref{new_mor}, let $(u,v),(u',v') : (e,\delta,\sigma) \looparrowright (e',\delta',\sigma')$ be two parallel intertwiners between bicategorical Mealy machines; a 2-cell $(\varphi,\psi) : (u,v)\To (u',v')$ consists of a pair of 2-cells $\varphi : u\To u'$, $\psi : v\To v'$ such that the identities in \eqref{diag_2} hold true.%:
\end{definition}
\begin{remark}
  When it is specialised to the monoidal case, \autoref{anew_twocel} yields the following notion: a 2-cell $(f,g) : (U,V)\To (U',V')$ as in \autoref{not_triv} consists of a pair of morphisms $f : U\to U'$ and $g : V\to V'$ subject to the conditions that the two squares in \eqref{diag_1} commute: intuitively speaking, in this particular case, the machine 2-cells correspond to pairs $(f,g)$ of $\clK$-morphisms such that both pairs $(E'\otimes I'\otimes f,E'\otimes f)$ and $(g\otimes E\otimes I,g\otimes E)$ form morphisms in the arrow category of $\clK$.
  %   \[
  % \vxy[@R=5mm@C=2.5mm]{
  %  & E'\otimes U' \ar@{=}[dd]|\hole &  & E'\otimes I'\otimes U' \ar@{->}[ll]_{d'\otimes U'} \ar@{=}[dd] \\
  % E'\otimes U \ar@{->}[ru]^{E'\otimes f} \ar@{=}[dd] &  & E'\otimes I'\otimes U \ar@{->}[ru]_{E'\otimes I'\otimes f} \ar@{->}[ll]^(.7){d'\otimes U} \ar@{=}[dd] &  \\
  %  & V'\otimes E &  & V'\otimes E\otimes I \ar@{->}[ll]|(.54)\hole \\
  % V\otimes E \ar@{->}[ru]^{g\otimes E} &  & V\otimes E \otimes I \ar@{->}[ru]_{g\otimes E\otimes I} \ar@{->}[ll]^{V\otimes d} &
  % }\]
\end{remark}
\begin{remark}
Let $\bbB$ be a bicategory; in \cite{Katis1997} the authors exploit the universal property of a bicategory $\Omega\bbB=\cate{Psd}(\bfN,\bbB)$ as the category of pseudofunctors, lax natural transformations and modifications with domain the monoid of natural numbers, regarded as a single object category. The typical object of $\Omega\bbB$ is an endomorphism $i : A\to A$ of an object $A\in\bbB$, and the typical 1-cell consists of a lax commutative square
\[\vxy{
  A \ar[r]\ar[d]\drtwocell<\omit>{}& A \ar[d]\\
  B \ar[r]& B.
}\]
This presentation begs the natural question of whether there is a tautological functor $\Mly_\bbB\to \Omega\bbB$ given by `projection', sending $(i,o;(e,\delta,\sigma))$ into $i$; the answer is clearly affirmative, and in fact such functor mates to a unique 2-functor $\bfN\boxtimes\Mly_\bbB\to \bbB$ under the isomorphism given by Gray tensor product \cite{Gray1974}; this somehow preserves the intuition (cf. \cite[§1]{Tierney1969}) of $\Omega\bbB$ as a category of `lax dynamical systems'.
\end{remark}

\section{Conclusions}\label{sec:concl}
We sketch some directions for future research.
\begin{conjecture}\label{in_tvprof}
  Given a monad $T$ on $\Set$ and a quantale $\clV$ \cite[Ch. 2]{Eklund2018} we can define the locally thin bicategory $\TVProf$ as in \cite[Ch. III]{hofmann2014monoidal}; as $(T,\clV)$ vary we generate a plethora of bicategories, yielding the categories of topological spaces, approach spaces \cite{lowen1997approach}, metric and ultrametric, and closure spaces as the \emph{$(T,\clV)$-categories} of \cite[§III.1.6]{hofmann2014monoidal}. We conjecture that when instantiated in $\TVProf$, \autoref{reach_eq} yields a 2-categorical way to look at topological, metric and loosely speaking `fuzzy' approaches to automata theory.
\end{conjecture}
\begin{conjecture}
  From \autoref{in_cat} and \ref{in_prof} we argue that the `non\hyp{}determinism via Kleisli category' approach of \cite{Guitart1980} can be carried over for the presheaf construction on $\Cat$ and its Kleisli bicategory $\Prof$: if automata (classically intended) in the Kleisli category of the powerset monad are nondeterministic automata in $\Set$, \emph{bicategorical} automata in the Kleisli \emph{bicategory} of the \emph{presheaf construction} (cf. \cite{Fiore2017}) are nondeterministic bicategorical automata: passing from \autoref{in_cat} to \autoref{in_prof} accounts for a form of non\hyp{}determinism.
  But then one might be able to address \emph{nondeterministic} bicategorical automata in $\bbB$ as \emph{deterministic} bicategorical automata in a generic proarrow equipment \cite{rosebrugh1988proarrows,wood1982abstract,wood1985proarrows} for $\bbB$!
\end{conjecture}

\bibliography{refs}

\begin{thebibliography}{10}

\bibitem{acc}
J.~Ad{\'a}mek, H.~Herrlich, and G.E. Strecker.
\newblock {\em Abstract and concrete categories: the {J}oy of cats}, volume~17
  of {\em Reprints in Theory and Applications of Categories}.
\newblock Theory and Applications of Categories, 2006.
\newblock Available online at
  \url{http://www.tac.mta.ca/tac/reprints/articles/17/tr17.pdf}.

\bibitem{17344305293823001d18e665ad1a9d7e3addbe67}
J.~Adámek.
\newblock Realization theory for automata in categories.
\newblock {\em Journal of Pure and Applied Algebra}, 9:281--296, 1977.
\newblock \href {https://doi.org/10.1016/0022-4049(77)90071-8}
  {\path{doi:10.1016/0022-4049(77)90071-8}}.

\bibitem{adam-trnk:automata}
J.~Adámek and V.~Trnková.
\newblock {\em Automata and Algebras in Categories}.
\newblock Kluwer, 1990.

\bibitem{lmcs:5252}
B.~Ahrens and P.~L. Lumsdaine.
\newblock {Displayed Categories}.
\newblock {\em {Logical Methods in Computer Science}}, {Volume 15, Issue 1},
  March 2019.
\newblock URL: \url{https://lmcs.episciences.org/5252}, \href
  {https://doi.org/10.23638/LMCS-15(1:20)2019}
  {\path{doi:10.23638/LMCS-15(1:20)2019}}.

\bibitem{arbib1975adjoint}
M.A. Arbib and E.G. Manes.
\newblock Adjoint machines, state-behavior machines, and duality.
\newblock {\em Journal of Pure and Applied Algebra}, 6(3):313--344, 1975.
\newblock \href {https://doi.org/10.1016/0022-4049(75)90028-6}
  {\path{doi:10.1016/0022-4049(75)90028-6}}.

\bibitem{Arbib1975}
M.A. Arbib and E.G. Manes.
\newblock A categorist's view of automata and systems.
\newblock In {\em Lecture Notes in Computer Science}, pages 51--64. Springer
  Berlin Heidelberg, 1975.
\newblock \href {https://doi.org/10.1007/3-540-07142-3_61}
  {\path{doi:10.1007/3-540-07142-3_61}}.

\bibitem{arbib1975fuzzy}
M.A. Arbib and E.G. Manes.
\newblock Fuzzy machines in a category.
\newblock {\em Bulletin of the Australian Mathematical Society},
  13(2):169--210, 1975.
\newblock \href {https://doi.org/10.1017/S0004972700024412}
  {\path{doi:10.1017/S0004972700024412}}.

\bibitem{Arbib1980}
M.A. Arbib and E.G. Manes.
\newblock Machines in a category.
\newblock {\em Journal of Pure and Applied Algebra}, 19:9--20, December 1980.
\newblock \href {https://doi.org/10.1016/0022-4049(80)90090-0}
  {\path{doi:10.1016/0022-4049(80)90090-0}}.

\bibitem{bainbridge1972unified}
E.S. Bainbridge.
\newblock A unified minimal realization theory, with duality, for machines in a
  hyperdoctrine.
\newblock In {\em Technical Report}. Computer and Communication Sciences
  Department, University of Michigan, 1972.

\bibitem{Bainbridge1975addressed}
E.S. Bainbridge.
\newblock Addressed machines and duality.
\newblock In Ernest~Gene Manes, editor, {\em Category Theory Applied to
  Computation and Control}, pages 93--98, Berlin, Heidelberg, 1975. Springer
  Berlin Heidelberg.

\bibitem{betti1981quasi}
R.~Betti and S.~Kasangian.
\newblock A quasi-universal realization of automata.
\newblock {\em Universit{\`a} degli Studi di Trieste. Dipartimento di Scienze
  Matematiche}, 1981.

\bibitem{bettiproperty}
R.~Betti and S.~Kasangian.
\newblock Una proprietà del comportamento degli automi completi.
\newblock {\em Universit{\`a} degli Studi di Trieste. Dipartimento di Scienze
  Matematiche}, 1982.

\bibitem{noi:completeness}
G.~Boccali, A.~Laretto, F.~Loregian, and S.~Luneia.
\newblock Completeness for categories of generalized automata.
\newblock mar 2023.
\newblock \href {https://arxiv.org/abs/2303.03867} {\path{arXiv:2303.03867}}.

\bibitem{ronnie}
R~Brown.
\newblock {\em Topology and groupoids}.
\newblock www.groupoids.org, rev., updated, and expanded version edition, 2006.

\bibitem{burroni2009distributive}
E.~Burroni.
\newblock Lois distributives. applications aux automates stochastiques.
\newblock {\em Theory and Applications of Categories}, 22:199--221, 2009.

\bibitem{benabou2000distributors}
J.~Bénabou and T.~Streicher.
\newblock Distributors at work.
\newblock Lecture notes written by Thomas Streicher, 2000.

\bibitem{cattani2005profunctors}
G.L. Cattani and G.~Winskel.
\newblock Profunctors, open maps and bisimulation.
\newblock {\em Mathematical Structures in Computer Science}, 15(03):553--614,
  2005.

\bibitem{colcombet_et_al}
T.~Colcombet and D.~Petrisan.
\newblock {Automata Minimization: a Functorial Approach}.
\newblock In {\em 7th Conference on Algebra and Coalgebra in Computer Science
  (CALCO 2017)}, volume~72 of {\em Leibniz International Proceedings in
  Informatics (LIPIcs)}, pages 8:1--8:16, Dagstuhl, Germany, 2017.
\newblock \href {https://doi.org/10.4230/LIPIcs.CALCO.2017.8}
  {\path{doi:10.4230/LIPIcs.CALCO.2017.8}}.

\bibitem{DArgenio2005}
P.R. D'Argenio and J.-P. Katoen.
\newblock A theory of stochastic systems part i: Stochastic automata.
\newblock {\em Information and Computation}, 203(1):1--38, November 2005.
\newblock \href {https://doi.org/10.1016/j.ic.2005.07.001}
  {\path{doi:10.1016/j.ic.2005.07.001}}.

\bibitem{2009}
M.~Droste, W.~Kuich, and H.~Vogler, editors.
\newblock {\em Handbook of Weighted Automata}.
\newblock Springer Berlin Heidelberg, 2009.
\newblock \href {https://doi.org/10.1007/978-3-642-01492-5}
  {\path{doi:10.1007/978-3-642-01492-5}}.

\bibitem{dubuc}
E.J. Dubuc.
\newblock {\em Kan extensions in enriched category theory}.
\newblock Lecture Notes in Mathematics, Vol. 145. Springer-Verlag, Berlin-New
  York, 1970.
\newblock \href {https://doi.org/10.1007/BFb0060485}
  {\path{doi:10.1007/BFb0060485}}.

\bibitem{Ehrig}
H.~Ehrig, K.-D. Kiermeier, H.-J. Kreowski, and W.~K{\"u}hnel.
\newblock {\em Universal theory of automata. {A} categorical approach}.
\newblock \href {https://doi.org/10.1007/978-3-322-96644-5}
  {\path{doi:10.1007/978-3-322-96644-5}}.

\bibitem{Eilenberg:1974}
S.~Eilenberg.
\newblock {\em Automata, Languages and Machines}, volume~A.
\newblock Academic Press, New York, 1974.

\bibitem{Eilenberg1967}
S.~Eilenberg and J.B. Wright.
\newblock Automata in general algebras.
\newblock {\em Information and Control}, 11(4):452--470, October 1967.
\newblock \href {https://doi.org/10.1016/s0019-9958(67)90670-5}
  {\path{doi:10.1016/s0019-9958(67)90670-5}}.

\bibitem{Eklund2018}
P.~Eklund, J.~Guti\'errez Garci\'a, U.~H\"ohle, and J.~Kortelainen.
\newblock {\em Semigroups in Complete Lattices}.
\newblock Springer International Publishing, 2018.
\newblock \href {https://doi.org/10.1007/978-3-319-78948-4}
  {\path{doi:10.1007/978-3-319-78948-4}}.

\bibitem{Fiore2017}
M.~Fiore, N.~Gambino, M.~Hyland, and G.~Winskel.
\newblock Relative pseudomonads, kleisli bicategories, and substitution
  monoidal structures.
\newblock {\em Selecta Mathematica}, 24(3):2791--2830, November 2017.
\newblock \href {https://doi.org/10.1007/s00029-017-0361-3}
  {\path{doi:10.1007/s00029-017-0361-3}}.

\bibitem{fiore}
T.M. Fiore.
\newblock {\em Pseudo Limits, Biadjoints, And Pseudo Algebras: Categorical
  Foundations of Conformal Field Theory}.
\newblock Memoirs AMS 860. American Mathematical Society, 2006.

\bibitem{goguen:realization}
J.A. Goguen.
\newblock Realisation is universal.
\newblock {\em Mathematical System Theory}, 6(4), 1973.

\bibitem{Goguen1975closed}
J.A. Goguen.
\newblock Discrete-time machines in closed monoidal categories. {I}.
\newblock {\em Journal of Computer and System Sciences}, 10(1):1--43, February
  1975.
\newblock \href {https://doi.org/10.1016/s0022-0000(75)80012-2}
  {\path{doi:10.1016/s0022-0000(75)80012-2}}.

\bibitem{Goguen1975}
J.A. Goguen, J.~W. Thatcher, E.~G. Wagner, and J.~B. Wright.
\newblock Factorizations, congruences, and the decomposition of automata and
  systems.
\newblock In {\em Lecture Notes in Computer Science}, pages 33--45. Springer
  Berlin Heidelberg, 1975.
\newblock \href {https://doi.org/10.1007/3-540-07162-8_665}
  {\path{doi:10.1007/3-540-07162-8_665}}.

\bibitem{Gray1974}
J.W. Gray.
\newblock {\em Formal Category Theory: Adjointness for 2-Categories}.
\newblock Springer Berlin Heidelberg, 1974.
\newblock \href {https://doi.org/10.1007/bfb0061280}
  {\path{doi:10.1007/bfb0061280}}.

\bibitem{guitart1974remarques}
R.~Guitart.
\newblock Remarques sur les machines et les structures.
\newblock {\em Cahiers de Topologie et Géométrie Différentielle
  Catégoriques}, 15:113--144, 1974.

\bibitem{guitart1978bimodules}
R.~Guitart.
\newblock Des machines aux bimodules.
\newblock Univ. Paris 7, apr 1978.
\newblock URL:
  \url{http://rene.guitart.pagesperso-orange.fr/textespublications/rg30.pdf}.

\bibitem{Guitart1980}
R.~Guitart.
\newblock Tenseurs et machines.
\newblock {\em Cahiers de topologie et g\'eom\'etrie diff\'erentielle},
  21(1):5--62, 1980.
\newblock URL: \url{http://www.numdam.org/item/CTGDC_1980__21_1_5_0/}.

\bibitem{ea309886425a9829ea5df96a513b157fd9940689}
R.~Guitart and L.~Van~den Bril.
\newblock D\'ecompositions et {Lax-compl\'etions}.
\newblock {\em Cahiers de topologie et g\'eom\'etrie diff\'erentielle},
  18(4):333--407, 1977.
\newblock URL: \url{http://www.numdam.org/item/CTGDC_1977__18_4_333_0/}.

\bibitem{hofmann2014monoidal}
D.~Hofmann, G.J. Seal, and W.~Tholen.
\newblock {\em Monoidal Topology: A Categorical Approach to Order, Metric, and
  Topology}, volume 153.
\newblock Cambridge University Press, 2014.
\newblock \href {https://doi.org/10.1017/CBO9781107517288}
  {\path{doi:10.1017/CBO9781107517288}}.

\bibitem{CLTT}
B.~Jacobs.
\newblock {\em Categorical Logic and Type Theory}.
\newblock Number 141 in SLFM. Elsevier, 1999.

\bibitem{Jacobs2006}
B.~Jacobs.
\newblock A bialgebraic review of deterministic automata, regular expressions
  and languages.
\newblock In {\em Algebra, Meaning, and Computation}, pages 375--404. Springer
  Berlin Heidelberg, 2006.
\newblock \href {https://doi.org/10.1007/11780274_20}
  {\path{doi:10.1007/11780274_20}}.

\bibitem{Jacobs2015}
B.~Jacobs.
\newblock New directions in categorical logic, for classical, probabilistic and
  quantum logic.
\newblock {\em Logical Methods in Computer Science}, 11(3):76, 2015.
\newblock Id/No 24.
\newblock \href {https://doi.org/10.2168/LMCS-11(3:24)2015}
  {\path{doi:10.2168/LMCS-11(3:24)2015}}.

\bibitem{jacobs_2016}
B.~Jacobs.
\newblock {\em Introduction to Coalgebra: Towards Mathematics of States and
  Observation}.
\newblock Cambridge Tracts in Theoretical Computer Science. Cambridge
  University Press, 2016.
\newblock \href {https://doi.org/10.1017/CBO9781316823187}
  {\path{doi:10.1017/CBO9781316823187}}.

\bibitem{justesen1968bikategorien}
M.~B. Justesen.
\newblock Bikategorien af profunktorer.
\newblock Naturvidenskabelig embedseksamen, Aarhus University, 1968.

\bibitem{kasangian1983cofibrations}
S.~Kasangian, G.M. Kelly, and F.~Rossi.
\newblock Cofibrations and the realization of non-deterministic automata.
\newblock {\em Cahiers de topologie et g{\'e}om{\'e}trie diff{\'e}rentielle
  cat{\'e}goriques}, 24(1):23--46, 1983.

\bibitem{kasangian1990glueing}
S.~Kasangian and R.~Rosebrugh.
\newblock Glueing enriched modules and composition of automata.
\newblock {\em Cahiers de topologie et g{\'e}om{\'e}trie diff{\'e}rentielle
  cat{\'e}goriques}, 31(4):283--290, 1990.

\bibitem{Katis1997}
P.~Katis, N.~Sabadini, and R.F.C. Walters.
\newblock Bicategories of processes.
\newblock {\em Journal of Pure and Applied Algebra}, 115(2):141--178, February
  1997.
\newblock \href {https://doi.org/10.1016/s0022-4049(96)00012-6}
  {\path{doi:10.1016/s0022-4049(96)00012-6}}.

\bibitem{ITA_2002__36_2_181_0}
P.~Katis, N.~Sabadini, and R.F.C. Walters.
\newblock Feedback, trace and fixed-point semantics.
\newblock {\em RAIRO - Theoretical Informatics and Applications - Informatique
  Th\'eorique et Applications}, 36(2):181--194, 2002.
\newblock URL: \url{http://www.numdam.org/articles/10.1051/ita:2002009/}, \href
  {https://doi.org/10.1051/ita:2002009} {\path{doi:10.1051/ita:2002009}}.

\bibitem{kelly1980unified}
G.M. Kelly.
\newblock A unified treatment of transfinite constructions for free algebras,
  free monoids, colimits, associated sheaves, and so on.
\newblock {\em Bulletin of the Australian Mathematical Society}, 22(01):1--83,
  1980.
\newblock \href {https://doi.org/10.1017/S0004972700006353}
  {\path{doi:10.1017/S0004972700006353}}.

\bibitem{kock1967limit}
A.~Kock.
\newblock {\em Limit monads in categories}.
\newblock The University of Chicago, 1967.

\bibitem{kock:cc-commnd}
A.~Kock.
\newblock Closed categories generated by commutative monads.
\newblock {\em J. Austral. Math. Soc.}, 12:405--424, 1971.

\bibitem{kock1972strong}
A.~Kock.
\newblock Strong functors and monoidal monads.
\newblock {\em Archiv der Mathematik}, 23(1):113--120, 1972.

\bibitem{Lack2008}
S.~Lack.
\newblock Note on the construction of free monoids.
\newblock {\em Applied Categorical Structures}, 18(1):17--29, October 2008.
\newblock \href {https://doi.org/10.1007/s10485-008-9167-y}
  {\path{doi:10.1007/s10485-008-9167-y}}.

\bibitem{openTransitionSystems21}
E.~Di Lavore, A.~Gianola, M.~Rom{\'{a}}n, N.~Sabadini, and P.~Soboci\'nski.
\newblock A canonical algebra of open transition systems.
\newblock In Gwen Sala{\"{u}}n and Anton Wijs, editors, {\em Formal Aspects of
  Component Software - 17th International Conference, {FACS} 2021, Virtual
  Event, October 28-29, 2021, Proceedings}, volume 13077 of {\em Lecture Notes
  in Computer Science}, pages 63--81. Springer, 2021.
\newblock \href {https://doi.org/10.1007/978-3-030-90636-8\_4}
  {\path{doi:10.1007/978-3-030-90636-8\_4}}.

\bibitem{coend-calcu}
F.~Loregian.
\newblock {\em Coend Calculus}, volume 468 of {\em London Mathematical Society
  Lecture Note Series}.
\newblock Cambridge University Press, first edition, jul 2021.
\newblock ISBN 9781108746120.

\bibitem{lowen1997approach}
R.~Lowen.
\newblock {\em Approach Spaces: The Missing Link in the
  Topology-uniformity-metric Triad}.
\newblock Oxford mathematical monographs. Clarendon Press, 1997.

\bibitem{Meseguer1990}
J.~Meseguer and U.~Montanari.
\newblock Petri nets are monoids.
\newblock {\em Information and Computation}, 88(2):105--155, October 1990.
\newblock \href {https://doi.org/10.1016/0890-5401(90)90013-8}
  {\path{doi:10.1016/0890-5401(90)90013-8}}.

\bibitem{Meseguer1975}
J.~Meseguer and I.~Sols.
\newblock Automata in semimodule categories.
\newblock In {\em Lecture Notes in Computer Science}, pages 193--198. Springer
  Berlin Heidelberg, 1975.
\newblock \href {https://doi.org/10.1007/3-540-07142-3_81}
  {\path{doi:10.1007/3-540-07142-3_81}}.

\bibitem{2d281e1e3e0b525128f55519cf8a03fa52ce6252}
G.~Naudé.
\newblock On the adjoint situations between behaviour and realization.
\newblock {\em Quaestiones Mathematicae}, 2:245--267, 1977.
\newblock \href {https://doi.org/10.1080/16073606.1977.9632546}
  {\path{doi:10.1080/16073606.1977.9632546}}.

\bibitem{NAUDE1979277}
G.~Naudé.
\newblock Universal realization.
\newblock {\em Journal of Computer and System Sciences}, 19(3):277--289, 1979.
\newblock \href {https://doi.org/10.1016/0022-0000(79)90005-9}
  {\path{doi:10.1016/0022-0000(79)90005-9}}.

\bibitem{Perrone2022}
P.~Perrone and W.~Tholen.
\newblock Kan extensions are partial colimits.
\newblock {\em Applied Categorical Structures}, 30(4):685--753, January 2022.
\newblock \href {https://doi.org/10.1007/s10485-021-09671-9}
  {\path{doi:10.1007/s10485-021-09671-9}}.

\bibitem{Pohl_1970}
I.~Pohl and M.A. Arbib.
\newblock Theories of abstract automata.
\newblock {\em Mathematics of Computation}, 24(111):760, jul 1970.
\newblock \href {https://doi.org/10.2307/2004866} {\path{doi:10.2307/2004866}}.

\bibitem{rosebrugh_sabadini_walters_1998}
R.~Rosebrugh, N.~Sabadini, and R.F.C. Walters.
\newblock Minimal realization in bicategories of automata.
\newblock {\em Mathematical Structures in Computer Science}, 8(2):93–116,
  1998.
\newblock \href {https://doi.org/10.1017/S0960129597002454}
  {\path{doi:10.1017/S0960129597002454}}.

\bibitem{rosebrugh1988proarrows}
R.~Rosebrugh and R.J. Wood.
\newblock Proarrows and cofibrations.
\newblock {\em Journal of Pure and Applied Algebra}, 53(3):271--296, 1988.
\newblock \href {https://doi.org/10.1016/0022-4049(88)90128-4}
  {\path{doi:10.1016/0022-4049(88)90128-4}}.

\bibitem{Rutten2000}
J.M. Rutten.
\newblock Universal coalgebra: a theory of systems.
\newblock {\em Theoretical Computer Science}, 249(1):3--80, October 2000.
\newblock \href {https://doi.org/10.1016/s0304-3975(00)00056-6}
  {\path{doi:10.1016/s0304-3975(00)00056-6}}.

\bibitem{Singh2021}
S.~Singh and S.P. Tiwari.
\newblock On the category of {L}-fuzzy automata, coalgebras and dialgebras.
\newblock {\em Fuzzy Sets and Systems}, 420:1--28, September 2021.
\newblock \href {https://doi.org/10.1016/j.fss.2020.07.013}
  {\path{doi:10.1016/j.fss.2020.07.013}}.

\bibitem{StreetFibreYoneda1974}
R.~Street.
\newblock Fibrations and {Y}oneda lemma in a 2-category.
\newblock In G.M. Kelly, editor, {\em Proceedings Sydney Category Theory
  Seminar 1972/1973}, volume 420 of {\em Lecture Notes in Mathematics}, pages
  104--133. Springer, 1974.
\newblock \href {https://doi.org/10.1007/BFb0063096}
  {\path{doi:10.1007/BFb0063096}}.

\bibitem{street1980fibrations}
R.~Street.
\newblock Fibrations in bicategories.
\newblock {\em Cahiers de topologie et géométrie différentielle
  catégoriques}, 21(2):111--160, 1980.

\bibitem{Tierney1969}
M.~Tierney.
\newblock {\em Categorical Constructions in Stable Homotopy Theory}.
\newblock Springer Berlin Heidelberg, 1969.
\newblock \href {https://doi.org/10.1007/bfb0101425}
  {\path{doi:10.1007/bfb0101425}}.

\bibitem{Venema2006}
Y.~Venema.
\newblock Automata and fixed point logic: A coalgebraic perspective.
\newblock {\em Information and Computation}, 204(4):637--678, April 2006.
\newblock \href {https://doi.org/10.1016/j.ic.2005.06.003}
  {\path{doi:10.1016/j.ic.2005.06.003}}.

\bibitem{merriam_webster}
VV.AA.
\newblock Fugal: definition \& meaning.
\newblock URL: \url{https://www.merriam-webster.com/dictionary/fugal}.

\bibitem{Walters1989}
R.F.C. Walters.
\newblock A note on context-free languages.
\newblock {\em Journal of Pure and Applied Algebra}, 62(2):199--203, December
  1989.
\newblock \href {https://doi.org/10.1016/0022-4049(89)90151-5}
  {\path{doi:10.1016/0022-4049(89)90151-5}}.

\bibitem{wood1982abstract}
R.J. Wood.
\newblock Abstract pro arrows {I}.
\newblock {\em Cahiers de topologie et g\'eom\'etrie diff\'erentielle},
  23(3):279--290, 1982.
\newblock URL: \url{http://www.numdam.org/item/CTGDC_1982__23_3_279_0/}.

\bibitem{wood1985proarrows}
R.J. Wood.
\newblock Proarrows {II}.
\newblock {\em Cahiers de Topologie et G\'eom\'etrie Diff\'erentielle
  Cat\'egoriques}, 26(2):135--168, 1985.
\newblock URL: \url{http://www.numdam.org/item/CTGDC_1985__26_2_135_0/}.

\end{thebibliography}
\appendix
\section{Appendix \thesection: Proofs}\label{sec:a:proofs}
\subsection{Diagrams}
\[\label{identizie}
\begin{tikzpicture}[scale=.625,baseline=(current bounding box.center)]
	\draw (0,0) |- (2,2) -- ++(-60:1.15470053838) -- (2,0) -- cycle;
	\draw (0,2) -- ++(-60:1.15470053838) -- (0,0);
	\draw (60:1.15470053838) -- ++(2,0);
	%===
	\node at (3,1) {$=$};
	\begin{scope}[xshift=3.5cm]
	\draw (0,0) |- (2,2) -- ++(-60:1.15470053838) -- (2,0) -- cycle;
	\draw (2,0) -- (2,2);
	\end{scope}
	\node[font=\footnotesize] at (.25,1) {$\delta$};
	\node[font=\footnotesize] at (5.5+.25,1) {$\delta'$};
	\node[font=\footnotesize] at (4.5,1) {$\epsilon$};
	\node[font=\footnotesize] at (1.25,.5) {$\epsilon$};
	\node[font=\footnotesize] at (1.25,1.5) {$\iota$};
	\begin{scope}[xshift=8cm]
		\draw (0,0) |- (2,2) -- ++(-60:1.15470053838) -- (2,0) -- cycle;
		\draw (0,2) -- ++(-60:1.15470053838) -- (0,0);
		\draw (60:1.15470053838) -- ++(2,0);
		%===
		\node at (3,1) {$=$};
		\begin{scope}[xshift=3.5cm]
		\draw (0,0) |- (2,2) -- ++(-60:1.15470053838) -- (2,0) -- cycle;
		\draw (2,0) -- (2,2);
		\end{scope}
		\node[font=\footnotesize] at (.25,1) {$\sigma$};
		\node[font=\footnotesize] at (5.5+.25,1) {$\sigma'$};
		\node[font=\footnotesize] at (4.5,1) {$\omega$};
		\node[font=\footnotesize] at (1.25,.5) {$\epsilon$};
		\node[font=\footnotesize] at (1.25,1.5) {$\iota$};
	\end{scope}
	\node at (7,1) {\small and};
	\node at (14.5,1) {\small ;};
\end{tikzpicture}
\]
\[\vxy{
	A \drtwocell<\omit>{\iota}\ar[r]^u\ar[d]_i & A'\ar[d]^{i'} & A \drtwocell<\omit>{\epsilon}\ar[r]^u\ar[d]_e & A'\ar[d]^{e'}  & A \drtwocell<\omit>{\omega}\ar[r]^u\ar[d]_o & A'\ar[d]^{o'} \\
	A \ar[r]_u & A' & B \ar[r]_v & B' & B \ar[r]_v & B'
}\label{cellette}\]
\[\label{diag_2}
    \begin{tikzpicture}[scale=.5,baseline=(current bounding box.center)]
      \commuThing{\iota}{\varphi}{\varphi}
      \begin{scope}[xshift=8cm]
        \commuThing{\epsilon}{\varphi}{\psi}
      \end{scope}
      \begin{scope}[xshift=16cm]
        \commuThing{\omega}{\varphi}{\psi}
      \end{scope}
    \end{tikzpicture}
  \]
\[\label{diag_1}\vxy{
	E'\otimes I'\otimes U\ar[r]^{d'\otimes U}\ar[d]_{E'\otimes I'\otimes f} & E'\otimes U \ar[d]^{E'\otimes f}& V\otimes E \otimes I \ar[r]^{V\otimes d}\ar[d]_{g\otimes E\otimes I}& V\otimes E \ar[d]^{g\otimes E}\\
	E'\otimes I'\otimes U' \ar[r]_{d'\otimes U'}& E'\otimes U' & V'\otimes E \otimes I \ar[r]_{V'\otimes d}& V'\otimes E
}\]

\subsection{Proofs}
\begin{proof}[Proof of \autoref{lemma_1}]\label{proof_of_lemma_1}
  In order to prove that the assignment $s\mapsto s^\flat$ is well defined in the set of fugal automata, we proceed by induction on the length of a string $\ell$. We have to prove that
  \[ s^{\flat}(e,\ell\concat as)=s^{\flat}(e,\ell)\concat s^{\flat}(d^*(e,\ell),as) \]
  The base case $\ell = [\,]$ is evidently true, so suppose that $\ell = x\kons xs$ is not empty and the claim is true for every choice of a shorter $xs$: then,
\begin{align*}
s^{\flat}(e, (x\kons xs)\concat as) &= s^{\flat}(e, (x\kons xs)\concat as)\\
&=s^{\flat}(e, x\kons (xs\concat as))\\
&=s(e,x)\kons s^\flat(d(e,x), xs\concat as)\\
&=s(e,x)\kons \big(s^\flat(d(e,x), xs) \concat s^\flat(d^*(e,xs), as)\big)\\
&=\big(s(e,x)\kons s^\flat(d(e,x), xs)\big) \concat s^\flat(d(x,d^*(e,xs)), as)\\
&= s^\flat(e, x\kons xs) \concat s^\flat(d^*(e,x\kons xs), as)\\
&= s^\flat(e, \ell) \concat s^\flat(d^*(e,\ell), as).
\end{align*}
We now have to show that any 2-cell $f:(E,d,s)\to (F,c,t)$ is in fact a 2-cell $(E,d^*,s^{\flat})\to (F,c^*,t^{\flat})$. This can be done by induction as well, with completely similar reasoning.%
\end{proof}
\begin{proof}[Proof of \autoref{is_beh}]\label{proof_of_is_beh}
We have to prove that
\[(s_2 \c s_1)^\flat = s_2^\flat \c s_1^\flat.\]
The two functions coincide on the empty list by definition; hence, let $\ell = a\kons as$ be a nonempty list and $(e,f)\in E\times F$ a generic element. The right\hyp{}hand side of the equation is
\begin{align*}
(s_2^\flat \c s_1^\flat)((e,f), a\kons as) &= s_2^\flat(f, s_1^\flat(e,a\kons as))\\
&=s_2^\flat\big(f, s_1(e,a)\kons s_1^\flat(d_1(e,a), as)\big)\\
&=s_2(f, s_1(e,a)) \kons s_2^\flat(d_2(f, s_1(e,a)),s_1^\flat(d_1(e,a), as))\\
&=(s_2\c s_1)((e,f),a) \kons (s_2\c s_1)^\flat((d_2\c d_1)((e,f),a),as)
\end{align*}
which concludes the proof.
\end{proof}
\begin{proof}[Proof of \autoref{fun_mly_mac}]\label{proof_of_fun_mly_mac}
% \todo[inline]{}
Similarly to \autoref{is_beh}, we have to prove that $d_2^*\c d_1^*=(d_2\c d_1)^*$ whenever $d_2,d_1$ are two dynamic maps of composable Mealy machines $\langle s_1 ,d_1\rangle : E\times M\to N\times E$ and $\langle s_2 ,d_2\rangle : F\times N \to P\times F$. This, together with \autoref{is_beh}, will establish functoriality on 1-cells of $(\firstblank)^\flat$. Functoriality on 2-cells is very easy to establish. For what concerns $\Pi$, the proof amounts to showing that the composition of (fugal) Mealy machines gets mapped into the composition of spans in $\Mac$; this can be checked with ease and follows from the fact that the translation category of the action $d_2\c d_1$ as defined in \eqref{d2d1_term} has the universal property of the pullback $\clZ$ in
\[\vxy[@C=4mm@R=4mm]{
  &&\clZ\ar[dr]\ar[dl]&&\\
&\clE[d_1^*]\ar[dr]^{\Sigma_1}\ar[dl]_{D_1}&&\clE[d_2^*]\ar[dr]^{\Sigma_2}\ar[dl]_{D_2}\\
M && N && R.
}\]
This is a straightforward check, and it is also straightforward to see that the composition of $\Sigma_2$ with the right projection from $\clZ$ coincides with the `Sigma' functor induced by $s_2\c s_1$, which concludes the proof.
\end{proof}
\begin{proof}[Proof of \autoref{the_adjunzia}]\label{proof_of_the_adjunzia}
It is worthwhile to recall what a biadjunction is
\[\vxy{
  F : \mathsf C \ar@<.45em>[r]\ar@{}[r]|\perp &\ar@<.45em>[l] \mathsf D : G
}\]
if $\mathsf C,\mathsf D$ are bicategories (cf. \cite[Ch. 9]{fiore}): for each two objects $C,D$ we are given an equivalence between hom\hyp{}categories $\mathsf D(FC,D)\simeq\mathsf C(C,GD)$, i.e. a pair of functors $H:\mathsf D(FC,D)\leftrightarrows\mathsf C(C,GD):K$ whose composition in both directions is isomorphic to the identity functor of the respective hom\hyp{}category --and all this depends naturally on $C,D$.

In order to prove this, let's fix a set $A$ and a monoid $M$, let's build functors
\[\vxy{
\bMly_\Set(A^*,M) \ar[r]^-H & \Mly_\Set(A,UM) & \Mly_\Set(A,UM) \ar[r]^-K & \bMly_\Set(A^*,M)
}\]
and prove that they form an equivalence of categories by explicitly showing that $HK$ and $KH$ are isomorphic to the respective identities. We'll often adopt the convenient notation $\langle s,d\rangle : E\times X\to Y\times E$ for a Mealy machine of input $X$ and output $Y$.
\begin{itemize}
\item Let $\langle s,d\rangle : E\times A^*\to M\times E$ be a fugal Mealy machine; $H\langle s,d\rangle$ is defined as the composition
\[\vxy{
  E\times A\ar[r]^-{E\times\eta_A} & E\times A^* \ar[r]^-{\langle s,d\rangle}& M\times E
}\]
where $\eta_A : A\to A^*$ is the unit of the free\hyp{}forgetful adjunction between $\Set$ and monoids. In simple terms, $H$ acts `restricting' a fugal Mealy machine to the set of generators of its input.
\item Let $\langle s_0,d_0\rangle : F\times A \to UM\times F$ be any Mealy machine on $\Set$, where $UM$ means that $M$ is regarded as a mere set; $K\langle s_0,d_0\rangle$ is defined as the composition
\[\vxy{
  F\times A^* \ar[r]^-{\langle s_0,d_0\rangle^\flat}& (UM)^* \times F \ar[r]^-{\varepsilon \times F} & M\times F
}\]
where $\varepsilon : (UM)^* \to M$ is the counit of the free\hyp{}forgetful adjunction between $\Set$ and monoids, and $\langle s_0,d_0\rangle^\flat$ is the fugal extension of \autoref{lemma_1}.
\end{itemize}
The claim is now that the fugal Mealy machine $KH\langle s,d\rangle$ coincides with $\langle s,d\rangle$, and that the generic Mealy machine $HK\langle s_0,d_0\rangle$ coincides with $\langle s_0,d_0\rangle$.

\medskip
Both statements depend crucially on the following fact: if $s : E\times M \to N$ satisfies Equation \eqref{beheq}, then for all $e\in E$ the element $s(1_M,e)$ is idempotent in $N$. In particular, if $N$ is free on a set $B$, $s(1_M,e)=[\,]$ is the empty list, and more in particular, for a generic Mealy machine $\langle s,\firstblank\rangle$ the fugal extension $s^\flat$ is such that for all $e\in E$, $s^\flat([\,],e)=[\,]$.

\medskip
Given this, observe that the Mealy machine $HK\langle s_0,d_0\rangle$ coincides with $\langle s_0^\flat \circ (F\times\eta_A),d_0^*\circ (F\times\eta_A)\rangle$; now clearly the composition $d_0^*\circ (F\times\eta_A)$ coincides with $d_0 : F\times A \to F$ and the two maps determine each other. As for $s_0^\flat \circ (F\times\eta_A)$, we have that for all $(f,a)\in F\times A$
\begin{align*}
s_0^\flat \circ (F\times\eta_A)(f,a) &= s_0^\flat(f,a\kons [\,]) \\
&= s_0(f,a)\kons s^\flat_0(f,[\,])\\
&= s_0(f,a) \kons [\,]
\end{align*}
Reasoning similarly, one proves that the fugal Mealy machine $KH\langle s,d\rangle$ has components $\langle (s\circ (E\times\eta_A))^\flat, (d\circ (E\times\eta_A))^*\rangle$: again, since functions $E\times A\to E$ correspond bijectively to monoid actions $E\times A^* \to E$, the map $(d\circ (E\times\eta_A))^*$ coincides with $d$; as for $(s\circ (E\times\eta_A))^\flat$, we can argue by induction that
\begin{align*}
(s\circ (E\times\eta_A))^\flat(e,[\,]) &= [\,] = s(e, [\,])\\
(s\circ (E\times\eta_A))^\flat(e, a\kons as) &= s(e,a)\kons (s\circ (E\times\eta_A))^\flat(d(a,e), as)\\
&=s(e,a)\kons s(d(a,e), as) \\
&=s(e,a\kons as)
\end{align*}
where the last equality uses that $s$ was fugal to start with. This concludes the proof.
\end{proof}
\begin{proof}[Proof of \autoref{mre_as_diag}]\label{proof_of_mre_as_diag}
The category $\clP$ is in fact 2-discrete (it has no 2-cells) and its objects and morphisms are arranged as follows:
\[\vxy{
    1 & \ar@<.25em>[l]^-y\ar@<-.25em>[l]_-x 0 \ar@<.25em>[r]^z\ar@<-.25em>[r]_t & 2
  }\]
For lack of a better name, $\clP$ is the \emph{generic double span}.

The functors $W,G$ are then constructed as follows:
\begin{itemize}
  \item $G : \clP\to\Cat$ is constant on objects at the category $\clK(A,B)$, and chooses the double span
        \[\vxy{
          \clK(A,B) & \ar@<.25em>[l]^-\id\ar@<-.25em>[l]_-\id\clK(A,B) \ar@<.25em>[r]^-{\firstblank\circ i} \ar@<-.25em>[r]_-\id & \clK(A,B);
          }\]
  \item $W : \clP\to\Cat$ chooses the double span
        \[\vxy{
          \{0\to 1\} & \ar@<.25em>[l]^-j\ar@<-.25em>[l]_-j \{\heartsuit,\spadesuit\} \ar@<.25em>[r]^-{c_0} \ar@<-.25em>[r]_-{c_1} & \{0\to 1\}
          }\]
        where $\{\heartsuit,\spadesuit\}$ is a discrete category with two objects, $j = \left(\bsmat \heartsuit\mapsto 0 \\ \spadesuit\mapsto 1\esmat\right)$, and $c_k$ is constant at $k\in\{0,1\}$.
\end{itemize}
Now, it is a matter of unwinding the definition of a natural transformation $\alpha : W \To G$ to find that we are provided with maps
\begin{align}
  \{e,\#\} & = \alpha_0 : W0 \to \clK(A,B)\notag \\
  \sigma   & = \alpha_1 : W1 \to \clK(A,B)       \\
  \delta   & = \alpha_2 : W2 \to \clK(A,B)\notag
\end{align}
and with commutative diagrams arising from naturality as follows, if we agree to label $\alpha_0(\heartsuit)=e$ and $\alpha_0(\spadesuit)=o$, and we blur the distinction between $\alpha_0$ and the embedding of its image $\{e,o\}$ in $\clK(A,B)$:
\[\small\vxy{
    \{e,o\}\ar[r]^-j\ar[d]_{\alpha_0} & \{0\to 1\}\ar[d]^{\alpha_1} & \{e,o\}\ar[r]^{c_0}\ar[d]_{\alpha_0} & \{0\to 1\}\ar[d]^{\alpha_2} & \{e,o\}\ar[r]^{c_1}\ar[d]_{\alpha_0} & \{0\to 1\}\ar[d]^{\alpha_2} \\
    \clK(A,B) \ar@{=}[r] & \clK(A,B) & \clK(A,B)\ar[r]_{\firstblank\circ i} & \clK(A,B) & \clK(A,B) \ar@{=}[r] & \clK(A,B)
  }\]
Altogether, we have that these data yield a diagram of 2-cells
\[\vxy{
  &A\ar[dl]_i\ar[dr]_e\druppertwocell^o{^\sigma}&\\
  A\ar[rr]_e&\utwocell<\omit>{\delta}&B   }\]
as in \eqref{mach_as_kan}. Modifications between these natural transformations correspond to suitable arrangements of 2-cells, in such a way that we recover the notion of morphism of bicategorical Moore machine given in \ref{bmo_2}.

In case the output $o$ is fixed, we just constrain $\alpha_0(\spadesuit)$ to be mapped in $o$ and modifications to be the identity at $\spadesuit$.

For bicategorical Mealy machines, redefine $Gx=Gz=\firstblank\circ i$ and the rest of the argument is unchanged.
\end{proof}

% \section{Appendix \thesection: Agda formalization}\label{sec:a:agda}
% \input{secs/07-formalization.tex}
\begin{discussion}\label{complain}
	In a world of war and crippling inflation bytes are expensive, so page limits shorten by the month. \emph{This forces authors to shrink their papers, and the only way to do that is remove text.}

	A simple interpolation suggests that \emph{one day, the average submission will consist of just the picture of a cat surrounded by a circle and a square}; already today, we feel constrained to push in the appendix the email addresses of the authors: \emails.
\end{discussion}
\end{document}